\documentclass[11pt,a4paper]{article}

\usepackage{amsthm,amsmath,amsfonts,amssymb}
\usepackage[boxed,linesnumbered]{algorithm2e}
\usepackage{color}
\usepackage{float}
\usepackage{graphicx}
\usepackage{hyperref}
\usepackage{latexsym,epsfig}

\def\UMR{\mathop{\rm UMR }\nolimits}

\title{Network reliability in hamiltonian graphs}
\author{Pol Llagostera, Nacho L\'opez, Carles Comas.\\
{\small Dep. de Matem\`atica, Universitat de Lleida}\\
 {\small Lleida, Spain}\\
 {\small {\tt pol.llagostera@udl.cat}}\\
 {\small {\tt \{nlopez,carles.comas\}@matematica.udl.cat}}
}

\newenvironment{example}{{\bf Example:}}{}
\newtheorem{proposition}{Proposition}[section]
\newtheorem{corollary}{Corollary}[section]
\newtheorem{lemma}{Lemma}[section]
\newtheorem{theorem}{Theorem}[section]
\newtheorem{conjecture}{Conjecture}[section]
\newtheorem{observation}{Observation}[section]
\newtheorem{problem}{Problem}[section]

\begin{document}

\maketitle

\begin{abstract}
The reliability polynomial of a graph gives the probability that a graph remains operational when all its edges could fail independently with a certain fixed probability. In general, the problem of finding uniformly most reliable graphs inside a family of graphs, that is, one graph whose reliability is at least as large as any other graph inside the family, is very difficult. In this paper, we study this problem in the family of graphs containing a hamiltonian cycle.

\end{abstract}

\section{Introduction}

\subsection*{Notation and terminology}

In the reliability context, networks are modeled by graphs. We recall that a graph is an ordered pair $G=(V,E)$, where $V$ is a non empty set of \emph{nodes} or \emph{vertices}, and $E$ is a set of unordered pairs of different elements of $V$, called \emph{links} or \emph{edges}. The degree of a vertex $v \in V$, denoted by $d(v)$, is the number of incident edges at $v$. A {\em walk\/} of length $\ell\geq 0$ from a vertex $u$ to a vertex $v$ is a sequence of $\ell+1$ vertices, $u_0u_1\dots u_{\ell-1}u_\ell$, such that $u=u_0$, $v=u_\ell$ and each pair $u_{i-1}u_i$, for $i=1,\ldots,\ell$, is an edge of $G$. A {\em connected} graph has always a walk between any pair of vertices. Otherwise, the graph is {\em not connected}. Our model is a stochastic network with perfect nodes and edge failures: each edge remains connected independently with probability $p$ (every edge has the same probability of being operational). Moreover, no repair is allowed after an edge fails. \\

Through this paper, we consider the problem of computing the probability that a network remains connected. More precisely, we focus on what is called, the {\em all-terminal network reliability problem}: given the probability $p$ of an edge being operational in a network $G$, what is the probability that there exists an operational path between every pair of nodes $u$ and $v$ of $G$.

\subsubsection*{Hamiltonian graphs}

A graph $G$ is a {\em hamiltonian graph} if it contains a spanning cycle, that is, a cycle passing through all the vertices of the graph. This cycle is called a {\em hamiltonian cycle}. The problem of finding a hamiltonian cycle takes back to the 1850s when Sir William Rowan Hamilton presented the problem of finding a hamiltonian cycle in a dodecahedron. The complexity of finding a hamiltonian cycle in a graph is in general NP-complete and much research in this area is devoted to finding necessary and/or sufficient conditions for a graph to be hamiltonian. One of such conditions is Ore's Theorem, which is a well-known result in the field.

\begin{theorem}[Ore]\label{the:Ore}
	Let $G$ be a connected graph of order $n$ such that $d(u)+d(v) \geq n$ for any two pair of non adjacent vertices $u$ and $v$. Then $G$ is a hamiltonian graph. 
\end{theorem}

Recent developments and many results regarding hamiltonian graphs can be found in \cite{Gould2003}.

\subsection*{The reliability polynomial}
 
Let $|G|$ denote the number of edges of a network $G=(V,E)$. Let us consider the set ${\cal G}$ of connected spanning subgraphs of $G$. Then, the probability that $G$ is connected, as a function of $p$, is

\begin{equation}  
\label{eq:relpoly1}
\sum_{\substack{G' \in {\cal G}}} p^{|G'|} (1-p)^{m-|G'|}
\end{equation}

This formula is known as {\em the reliability polynomial of $G$}, and it is denoted as $\mbox{Rel}(G,p)$. There are several methods for computing $\mbox{Rel}(G,p)$, but in general this problem is NP-complete (see \cite{Perez2018}). A \emph{pathset} of a graph $G=(V,E)$ is a subset $N \subseteq E$ of edges that makes the graph $(V,N)$ connected. Hence, an alternative definition for the reliability polynomial is,
\begin{equation} 
\label{eq:relpoly2}
\mbox{Rel}(G,p) = \sum_{i=0}^{m} N_i p^i (1-p)^{m-i}
\end{equation}
where $N_i$ denotes the number of pathsets of cardinality $i$. From this point of view, some of the coefficients of the reliability polynomial are `easy' to compute. For instance, any spanning tree of $G$ contain $m-n+1$ edges, so $N_{m-n+1}=\tau$ 
where $\tau$ is the number of spanning trees of $G$ (also known as the tree-number).  Besides, $N_i= 0$ 
for all $i<m-n+1$. 
Also $N_i=\binom{m}{m-i}$ 
for all $i > m- \lambda$ 
, where $\lambda$ denotes the edge-connectivity of $G$. 

\subsection*{Uniformly most reliable graphs}

A main problem in the reliability context is concerned with the design of networks with `high' reliability. To this end, let ${\cal G}(n,m)$ be the set of all simple connected graphs with $n$ vertices and $m$ edges. Given two graphs $G,G' \in {\cal G}(n,m)$, we say that $G$ is {\em uniformly more reliable} than $G'$ if $\mbox{Rel}(G,p) \geq \mbox{Rel}(G',p)$ for all $p \in [0,1]$. This means that, for any value of an edge to being operational $p$, graph $G$ has higher probability to remain connected than graph $G'$. If there exists a graph $G$ such that $\mbox{Rel}(G,p) \geq \mbox{Rel}(H,p)$ for all $p \in [0,1]$ and for all $H \in {\cal G}(n,m)$, then $G$ is known as a {\em uniformly most reliable} graph in the set ${\cal G}(n,m)$. We point out that uniformly most reliable graphs have been also known as {\em uniformly optimally reliable} graphs (see \cite{Boesch88,Boesch91,Gross98,Smith90}) and {\em uniformly optimal} digraphs for the directed case (see \cite{Brown07}).

However, the reliability polynomial does not define a total ordering in ${\cal G}(n,m)$, because there are graphs whose corresponding reliability polynomial have a crossing point in the interval $(0,1)$ (see \cite{Colbourn87}). In order to prove that a graph is uniformly most reliable graph, the following observation is wide used. 
\begin{observation}\label{obs:rel}
Let $G,G'$ be graphs such that $N_i(G) \leq N_i(G')$, 
for all $0 \leq i \leq m$. Then $\mbox{Rel}(G,p) \leq \mbox{Rel}(G',p)$ for all $p \in [0,1]$.
\end{observation}
Hence, maximizing the number of pathsets $N_i$ 
has been a classical method to obtain uniformly most reliable graphs.\\

There exist uniformly most reliable graphs for $m \leq n+3$, and other simple values (see Table \ref{tab:mostReliable}), but in the other hand, there are infinitely many values of $n$ and $m$ where uniformly most reliable graphs do not exist (see \cite{Myr91}).  The uniformly most reliable graphs for $n \leq m \leq n+3$ are {\em subdivisions} of named graphs. The {\em edge subdivision} operation for an edge $uv$ is the deletion of $uv$ from the graph and the addition of two edges $uw$ and $wv$ along with the new vertex $w$. A graph which has been derived from $G$ by a sequence of edge subdivision operations is called a {\em subdivision} of $G$. For instance, every uniformly most reliable graph in the case $m=n+1$ ($n\geq 5$) is constructed taking a graph of order 5 as a basis and subdividing edges and adding vertices one by one by following the sequence $A,B,C,A,B,C,\dots$ (see Fig. \ref{fig:mrg}). This graph is also called the {\em Monma graph} \cite{WangWu74}. Uniformly most reliable graphs can be also obtained as subdivisions of named graphs for $m=n+2$ and $m=n+3$. We will denote these external graphs as $\UMR(n,m)$. It has been conjectured that $\UMR(n,n+4)$ are particular subdivisions of the Wagner graph for $n\geq8$ (see \cite{Rom17}). Despite these general constructions, uniformly most reliable graphs have been found only for $(n,m)=(8,12)$ (Wagner graph) and $(10,15)$ (Petersen graph).

\begin{table}[htb]
\begin{center}
\begin{tabular}{|c|l|c|}
\hline
$(n,m)$ & Uniformly most reliable graphs & Reference\\ \hline \hline
$(n,n-1)$ & $\UMR(n,n-1)$ (Tree graphs $T_n$) & \\ \hline
$(n,n)$ & $\UMR(n,n)$ (Cycle graphs $C_n$) & \\ \hline
$(n,n+1)$ & $\UMR(n,n+1)$ (Monma graph) & \cite{WangWu74} \\ \hline
$(n,n+2)$ & $\UMR(n,n+2)$ (subdivision of $K_4$) & \cite{Satya92}\\ \hline
$(n,n+3)$ & $\UMR(n,n+3)$ (subdivision of $K_{3,3}$) & \cite{Boesch91}\\ \hline
$(8,12)$ & Wagner graph & \cite{Rom17} \\ \hline
$(10,15)$ & Petersen graph & \cite{Rom18} \\ \hline
\hline
\end{tabular}
\end{center}
\caption{Uniformly most reliable graphs known in ${\cal G}(n,m)$, where $n$ denotes the number of vertices and $m$ the number of edges.}\label{tab:mostReliable}
\end{table}

\begin{figure}[htb]
\begin{center}
\begin{tabular}{cc}
\includegraphics[scale=0.5]{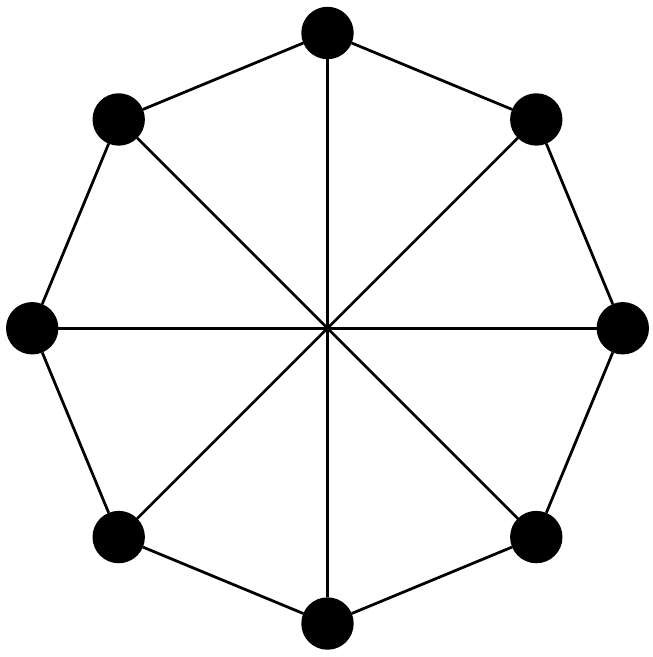} & \includegraphics[scale=0.5]{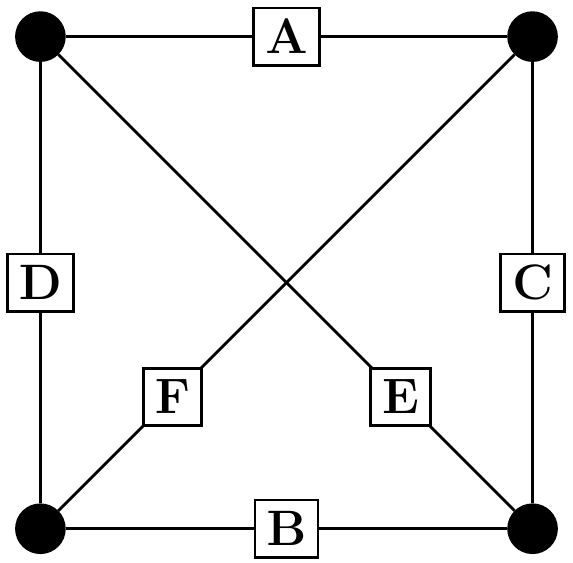} \\
 (a) Wagner graph. & (b) $\UMR(n,n+2)$.\\
 & \\
\includegraphics[scale=0.5]{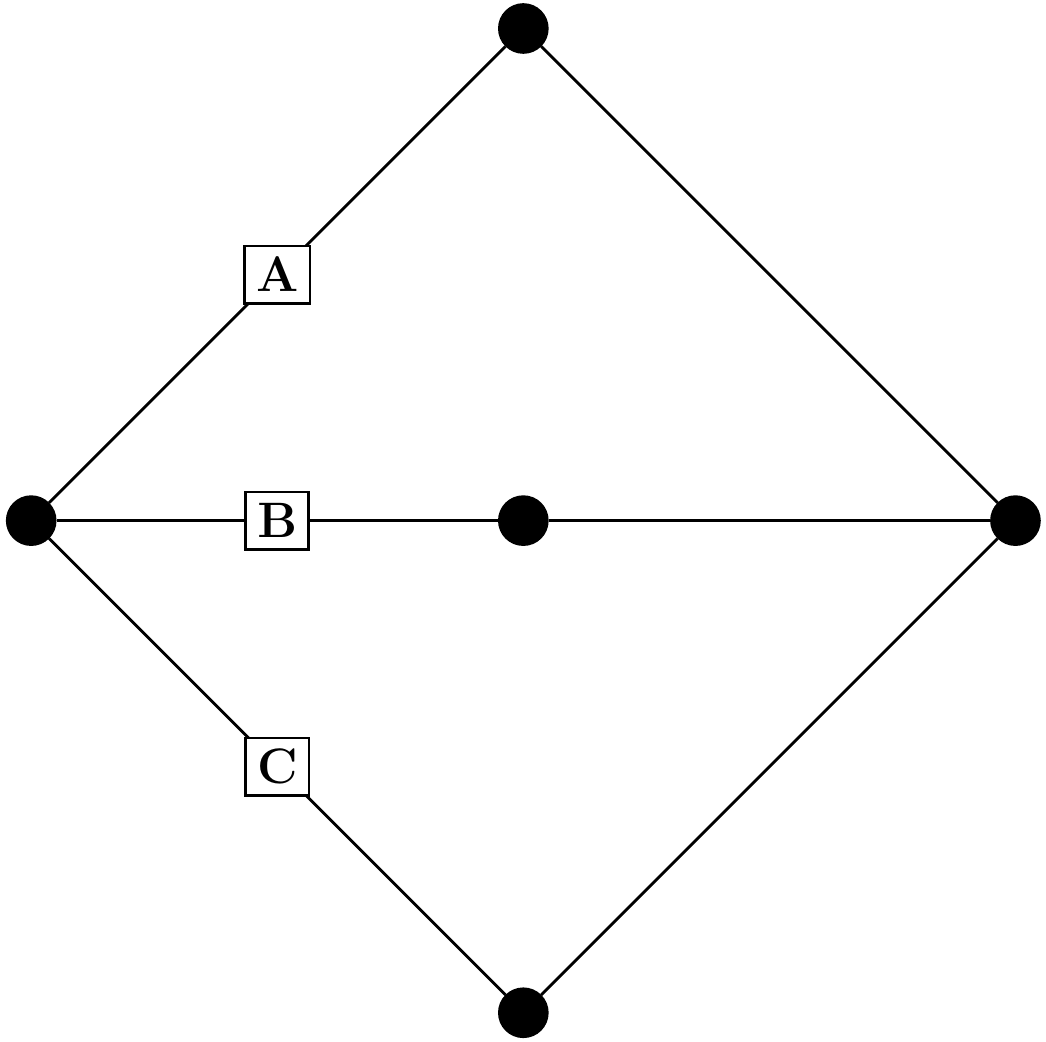} & \includegraphics[scale=0.5]{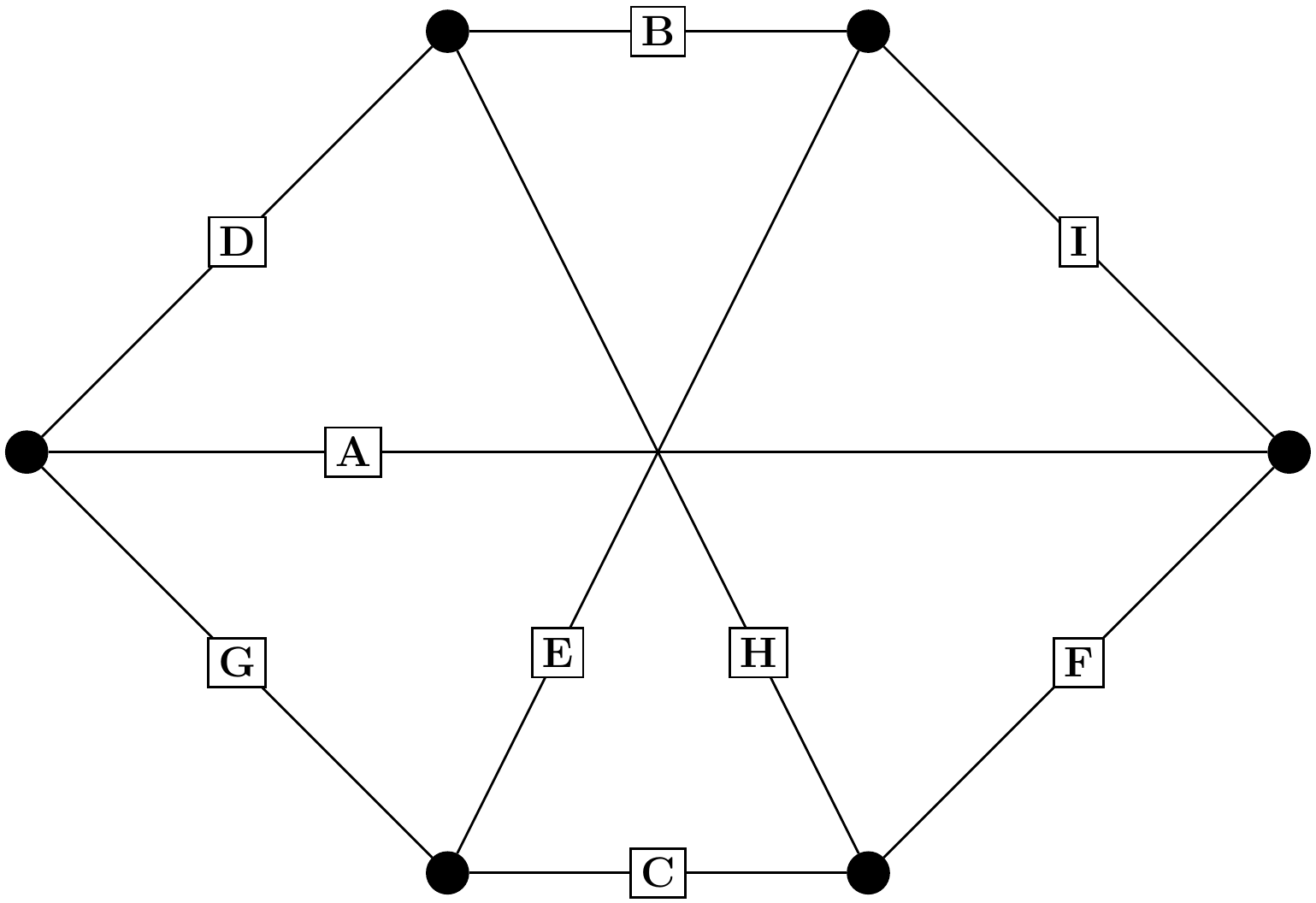} \\
(c) $\UMR(n,n+1)$. & (d) $\UMR(n,n+3)$.\\
\end{tabular}
\end{center}
\caption{Uniformly most reliable graphs.}\label{fig:mrg}
\end{figure}


\section{Uniformly most reliable hamiltonian graphs}\label{sec:most_reliable_ham}
Uniformly most reliable graphs have been classically studied on the set ${\cal G}(n,m)$ of (simple) graphs on $n$ vertices and $m$ edges. A restricted version of the problem would be the study of these external graphs inside an important family of graphs $\cal{F}$, such as the hamiltonian graphs. 
Let us denote as ${\cal H}(n,m)$ the set of (non-isomorphic) hamiltonian graphs with $n$ vertices and $m$ edges. Of course, ${\cal H}(n,m)$ is a subset of ${\cal G}(n,m)$ and since every hamiltonian graph on $n$ vertices and $m$ edges contains a spanning cycle, then we may assume that $m \geq n$. If a particular graph $G$ is uniformly most reliable in ${\cal G}(n,m)$ and it is in addition hamiltonian, then $G$ is also uniformly most reliable in ${\cal H}(n,m)$. This happens trivially in case $(n,n)$ for all $n\geq 3$, where $C_n$ are uniformly most reliable graphs, and also in $(8,12)$, where the Wagner graph is a uniformly most reliable graph and it is also hamiltonian (see Figure \ref{fig:mrg} (a)). It is well known that the Petersen graph is not hamiltonian (see \cite{Gould2003}) and in general, where $n$ is large enough, uniformly most reliable graphs are not hamiltonian, as next result shows. 

\begin{proposition}
Let us consider the uniformly most reliable graphs $\UMR(n,m)$ in ${\cal G}(n,m)$ for $m \leq n+3$. Then,
\begin{itemize}
 \item[(a)] $\UMR(n,n+1)$ is not hamiltonian;
 \item[(b)] $\UMR(n,n+2)$ is hamiltonian if and only if $n \leq 8$;
 \item[(c)] $\UMR(n,n+3)$ is hamiltonian if and only if $n \leq 13$;
\end{itemize}
\end{proposition}\label{prop:haminonham}
\begin{proof}
For any graph $G$ containing a vertex $v$ of degree $d(v) \geq 3$ and such that there are at least three different neighbors  $w_1,w_2,w_3$ of $v$ such that $d(w_1)=d(w_2)=d(w_3)=2$, then $G$ is not hamiltonian. Indeed, any vertex of $G$ uses exactly two edges in any hamiltonian cycle, hence the only two incident edges of $w_1,w_2$ and $w_3$ belong to any hamiltonian cycle. But one of these two pairs of edges is incident also to $v$, and hence $v$ would use three edges in any hamiltonian cycle, which is impossible. The graph $\UMR(n,n+1)$ given in Figure \ref{fig:mrg} has a vertex $v$ of degree $3$ and that all its neighbors have degree $2$. Hence it is not hamiltonian. It is easy to find a hamiltonian cycle in $\UMR(n,n+2)$ for $4 \leq n \leq 8$, when the four subdivisions $A,B,C,D$ are performed at most. But for $n \geq 9$ there exist a vertex of degree $3$, such that all its neighbors have degree $2$, and hence it is no longer hamiltonian. The argument for case (c) is similar: the sequence of edges (subdivided or not) $A,D,B,E,C,F$ gives a hamiltonian cycle for every $n \leq 12$ and the next subdivision (edge G) produces a vertex of degree $3$ such that all its neighbors have degree $2$. 
\end{proof}

Uniformly most reliable graphs are not hamiltonian when its order is large enough. So the natural question about which hamiltonian graphs are uniformly most reliable (whenever they do exist) arises. We partially answer this question in the following sections. 

\subsection{The case $m=n+1$}

Let us consider $G$ as a hamiltonian graph with $n$ vertices $\{0,1,\dots,n-1\}$ and $m=n+1$ edges where the hamiltonian cycle is given by the sequence $0,1,\dots,n-1,0$. Without loss of generality, we assume that the remaining edge of the graph joins vertex $0$ and vertex $x_1$. We may consider that $2 \leq x_1 \leq \lfloor \frac{n}{2} \rfloor$ (the remaining cases $\lfloor \frac{n}{2} \rfloor < x_1 < n-2$ produce an isomorphic graph taking $x_1'=n-x_1$). Then, the reliability polynomial of $G$ is given by
\begin{equation}\label{eq:eq1}
\mbox{Rel}(G,p)=p^m+mp^{m-1}(1-p)+\tau p^{m-2}(1-p)^2,
\end{equation}
where $\tau=n+x_1(n-x_1)$. Indeed, the removal of one edge or less guarantee the connectivity of the graph (the edge-connectivity of $G$ is precisely $\lambda=2$ for $n\geq 4$). Besides, deleting any set of three edges or more disconnects the graph ($N_i=0$ for all $i<m-2$). Hence, just the coefficient $N_{m-2}$ is unknown. This is precisely the tree number $\tau$. The number of spanning trees of $G$ is $n+x_1(n-x_1)$ since we have two cases: if edge $(0,x_1)$ is removed, then we can remove any of the $n$ remaining edges. Otherwise we may delete two edges belonging to the hamiltonian cycle. In order to guarantee the connection of the resulting graph, we must choose one edge from the cycle $0,1,\dots,x_1,0$ and the other from the cycle $0,x_1,x_1+1,\dots,n,0$. There are $x_1$ edges in one cycle and $n-x_1$ in the other. This gives the total number $n+x_1(n-x_1)$.

\begin{proposition}\label{prop:n+1}
The set of hamiltonian graphs ${\cal H}(n,n+1)$, $n \geq 4$, is totally ordered by the reliability polynomial and the uniformly most reliable graph is given by joining any of two vertices of $C_n$ at maximum distance.
\end{proposition}
\begin{proof}
For any graph $G$ in ${\cal H}(n,n+1)$, there exist $x_1$, with $2 \leq x_1 \leq \lfloor \frac{n}{2} \rfloor$ such that $G$ is isomorphic to a graph with hamiltonian cycle $0,1,\dots,n-1,0$ plus an edge $(0,x_1)$. According to \eqref{eq:eq1}, for any $p$, the value of $\tau(G)=-x_1^2+nx_1+n$ increases with $x_1$ from $2$ to $\lfloor \frac{n}{2} \rfloor$. The maximum of $\tau$ can be computed also with $\frac{\partial \tau}{\partial x_1}=0$ which gives $x_1=\frac{n}{2}$. Since $x_1$ must be an integer, for $n$ odd, the maximum is at $x_1=\lfloor \frac{n}{2} \rfloor$. This graph $G$ is isomorphic to the graph constructed from $C_n$ by joining two vertices at maximum distance.
\end{proof}

Although the most reliable graph constructed from $C_n$ and adding one single edge was previously known (see \cite{Rom17}) it is good to notice that ${\cal H}(n,n+1)$ is totally ordered by the reliability polynomial for all $n \geq 4$, since in general, this is not true for $m \geq n+2$ where reliability polynomials can cross for $p \in (0,1)$.




\subsection{The case $m=n+2$}\label{sec:n+2}

Here we will present uniformly most reliable graphs in ${\cal H}(n,n+2)$. The reliability polynomial for any $G \in {\cal H}(n,n+2)$ is,
\begin{equation}\label{eq:eq2}
\mbox{Rel}(G,p)=p^{m}+mp^{m-1}(1-p)+N_{m-2}p^{m-2}(1-p)^2+\tau p^{m-3}(1-p)^3.
\end{equation}
Any hamiltonian graph $G$ of order $n$ and size $n+2$ can be graphically depicted in a circular embedding where every vertex of $G$ is in a hamiltonian cycle $C$ of $G$ together with $2$ more edges ({\em chords}) through the hamiltonian cycle $C$. Then, the hamiltonian cycle is split by the chords in four paths of lengths $x_1,x_2,x_3$ and $x_4$ (see Figure \ref{fig:mplus2} A), where $x_1+x_2+x_3+x_4=n$. In fact, we have two more situations regarding the relative position between the chords (see Figure \ref{fig:mplus2} \^{A} and B). Notice that graphs of type \^{A} can be seen as a particular case of graphs of type A when one path length is zero. In any case, for any graph $G \in {\cal H}(n,n+2)$ of any type A or B, we associate its corresponding vector $(x_1,x_2,x_3,x_4)$. We will refer this vector as the vector of {\em c-path lengths}. In the other way around, given any positive integer vector $(x_1,x_2,x_3,x_4)$, such that $x_1+x_2+x_3+x_4=n$, we can construct any graph $G$ of ${\cal H}(n,n+2)$ of any type. Moreover, since any cyclic permutation of $(x_1,x_2,x_3,x_4)$ produces an isomorphic graph, we can just consider those vectors modulo cyclic permutations.
\begin{figure}[htb]
	\begin{center}
		\begin{tabular}{ccc}
			\includegraphics[scale=0.6]{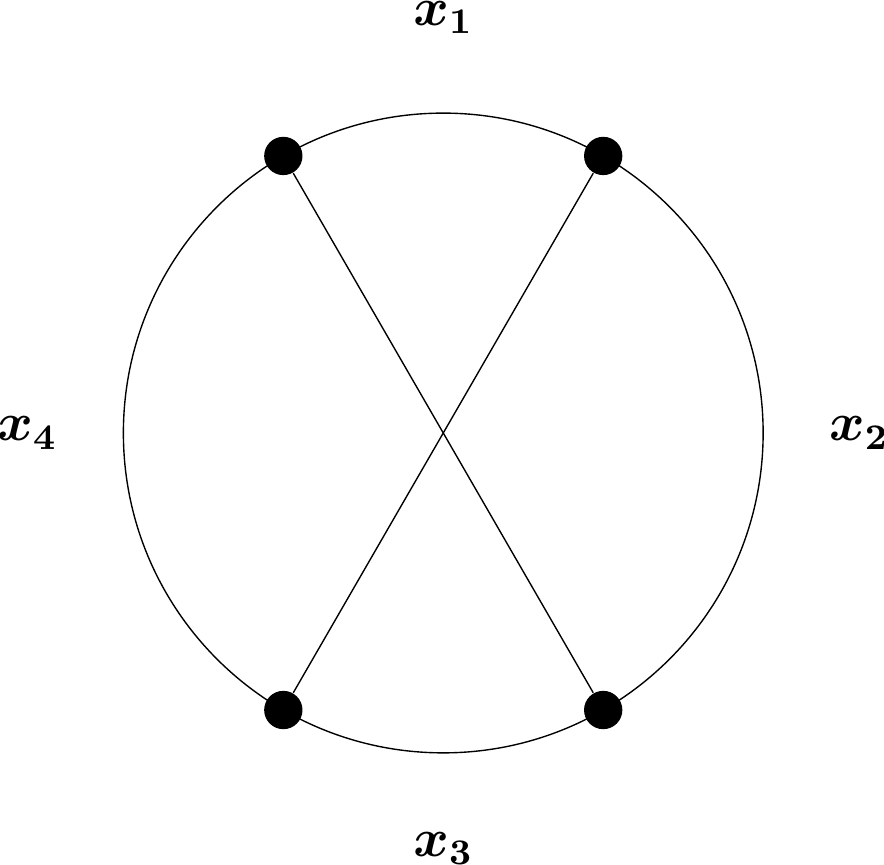} & \includegraphics[scale=0.6]{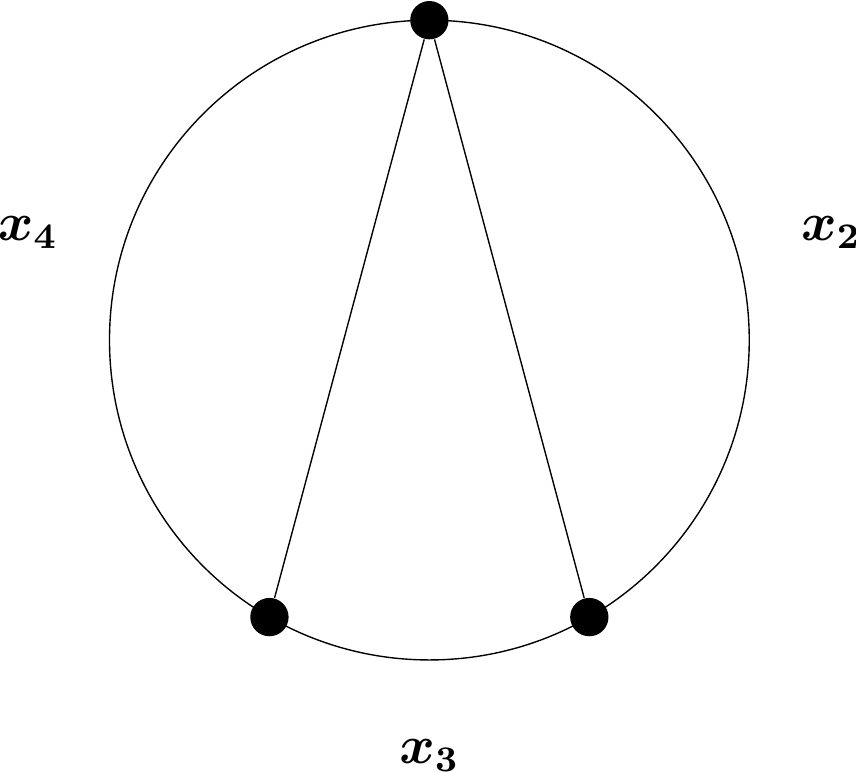} & \includegraphics[scale=0.6]{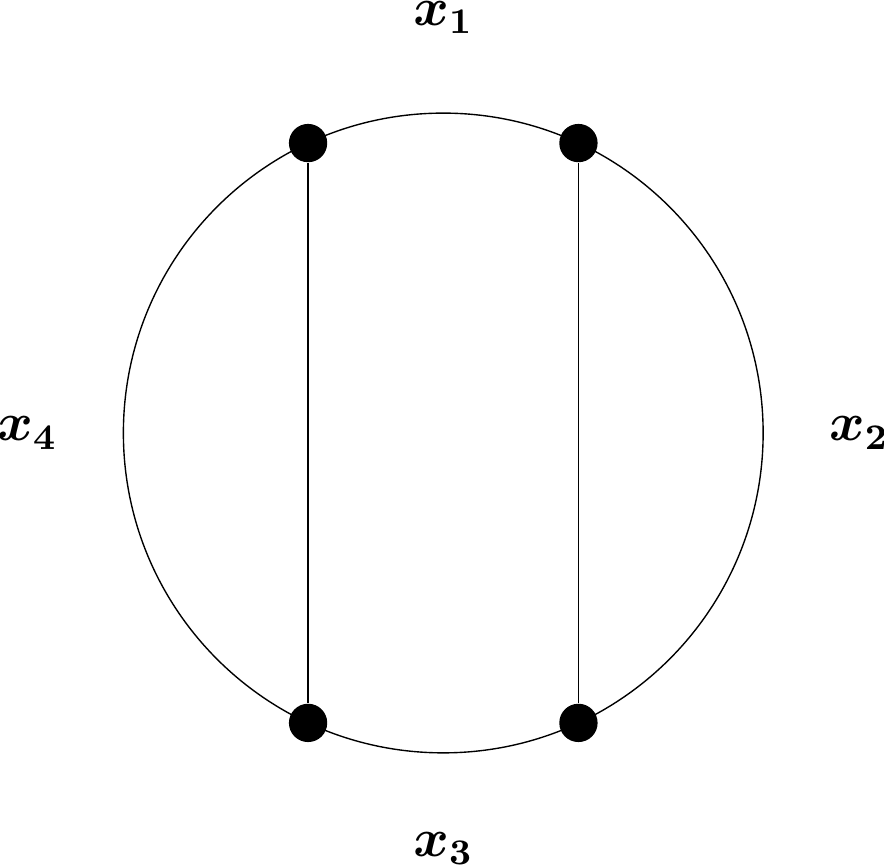} \\
			type A & type \^{A} & type B\\
		\end{tabular}
	\end{center}
	\caption{Different types of hamiltonian graphs with two chords.}\label{fig:mplus2}
\end{figure}

Next we will compute coefficients $N_{m-2}$ and $\tau$ in Eq. \eqref{eq:eq2} depending on $(x_1,x_2,x_3,x_4)$.

\begin{proposition}\label{prop:relm-2}
	Let $G$ be a graph in ${\cal H}(n,n+2)$ of type A with vector of c-path lengths $(x_1,x_2,x_3,x_4)$. Then,
	\begin{equation}\label{eq:Nm-2}
	\begin{array}{rcl}
	N_{m-2} & = & \displaystyle{1+2n+\sum_{1\leq i<j\leq 4}x_ix_j}.\\
	\tau & = & \displaystyle{n+(x_1+x_2)(x_3+x_4)+(x_1+x_4)(x_2+x_3)+\sum_{1\leq i<j<k\leq 4}x_ix_jx_k}.
	\end{array}
	\end{equation}

\end{proposition}
\begin{proof}
	The coefficient $N_{m-2}$ counts the number of connected graphs after the deletion of two edges of $G$. If these two edges are the chords, then we have just one graph (the hamiltonian cycle). Besides, if the removed edges are one chord and one edge of the cycle, then, we can choose between $2$ chords and $n$ edges of the hamiltonian cycle. This gives $2n$ graphs. Finally, if the removed edges belong to the hamiltonian cycle, then both edges must be from different paths defined by the chords. This gives the number $\sum_{1\leq i<j\leq 4}x_ix_j$. The computation of $\tau$ is similar, we need to remove three edges of $G$ in order to obtain a spanning tree of $G$. We proceed depending of which edges we are dealing with:
	\begin{enumerate}
		\item {\em Two chords plus one edge of the cycle}: Any edge of the cycle can be removed after the removal of the two chords, so we count $n$ spanning trees of this type.
		\item {\em One chord plus two edges of the cycle}: After the removal of one chord, say $c_1$, we must remove one edge of each path defined by the other chord. Since these paths have lengths $x_1+x_2$ and $x_3+x_4$, the number of spanning trees is $(x_1+x_2)(x_3+x_4)$. The same applies for the other chord, say $c_2$, where now the paths lengths are $x_1+x_4$ and $x_2+x_3$.
		\item {\em Three edges of the cycle}: The graph remain connected after the deletion of three edges of the cycle only if these three edges belong to different paths defined by the chords. Hence the number of spanning trees in this case is given by the all different three products of $x_1,x_2,x_3,x_4$.
	\end{enumerate}
\end{proof}
Graphs of type B are not interesting for the study of uniformly most-reliable hamiltonian graphs, as next result states:
\begin{proposition}\label{prop:propAB}
	Let $(x_1,x_2,x_3,x_4)$ be a c-path lengths vector of $G_A$ and $G_B$, which are graphs of types A and B, respectively. Then, $\tau(G_A)> \tau(G_B)$ and $N_{m-2}(G_A) \geq N_{m-2}(G_B)$. 
\end{proposition}
\begin{proof}
	Following the ideas behind the proof of proposition \ref{prop:relm-2} one can see that
	$$
	\begin{array}{l}
	\displaystyle{N_{m-2}(G_B)=1+2n+\big(\sum_{1\leq i<j\leq 4}x_ix_j\big)-x_1x_3,} \\
	\displaystyle{\tau(G_B)=n+(x_1+x_2+x_4)x_3+(x_2+x_4+x_3)x_1+\sum_{1\leq i<j<k\leq 4}x_ix_jx_k}. 
	\end{array}
	$$
	A simple comparison of both formulas with the ones given in Proposition \ref{prop:relm-2} gives the desired result, 
	$$
	\begin{array}{l}
	N_{m-2}(G_A)=N_{m-2}(G_B)+x_1x_3 \geq N_{m-2}(G_B) \textrm{ and } \\
	\tau(G_A)=\tau(G_B)+2x_2x_4>\tau(G_B).
	\end{array}
	$$
\end{proof}

In order to obtain those integer vectors ${\bf x}=(x_1,x_2,x_3,x_4)$ such that produce uniformly most-reliable hamiltonian graphs, we should maximize both $N_{m-2}$ and $\tau$ in equation \eqref{eq:eq2}. To this end, let us consider the set of c-path length vectors $X_n=\{(x_1,x_2,x_3,x_4) \in \mathbb{Z}^4_+ \, | \, x_1+x_2+x_3+x_4=n \}$ and the functions
\begin{equation}
\begin{array}{c@{\qquad \text{and} \qquad}c}
\begin{array}{cccc}
f\colon & X_n & \longrightarrow & \mathbb{Z} \\
& {\bf x} & \longmapsto & \tau
\end{array}
&
\begin{array}{cccc}
g\colon & X_n & \longrightarrow & \mathbb{Z} \\
& {\bf x} & \longmapsto & N_{m-2}
\end{array}
\end{array}
\label{eq:functfg}
\end{equation}
where $\tau$ and $N_{m-2}$ are given in Prop. \ref{prop:relm-2}.
Let $n=4k+\alpha$, $\alpha \in \{0,1,2,3\}$. We also define 
\[
D({\bf x})=\sum_{i=1}^4 |x_i-k|
\]
as a measure of closeness of any ${\bf x}=(x_1,x_2,x_3,x_4) \in X_n$ to the constant vector $(k,k,k,k)$. Then the elements of $X_n$ can be measured according to its closeness to this constant vector. For instance, when $\alpha=1$, there is just one vector in $X_n$ with $D=1$, which is $(k+1,k,k,k)$ (any cyclic permutation of this vector of c-path lengths gives an isomorphic graph, so we do not take them into account). Notice that $D$ must be odd when $\alpha=1$, so next $D$ is $3$ and the set of vectors of c-path lengths with $D=3$ is $\{(k+1,k+1,k-1,k),(k,k+2,k-1,k),(k,k+2,k,k-1)\}$. \\


The following mappings $\sigma$ and $\omega$ are a useful tool for our purposes: given a graph with vector of c-path lengths ${\bf x}=(x_1,x_2,x_3,x_4) \in X_n$, the mapping $\sigma$ moves just one vertex of a chord in such a way that two contiguous path lengths are modified by one unity each (see Fig. \ref{fig:sigma_omega}), that is, $\sigma({\bf x})=(x_1,x_2+1,x_3-1,x_4)$. Besides $\omega$ moves two vertices corresponding to different chords in such a way that two non-contiguous path lengths are modified by one unity each, that is, $\omega({\bf x})=(x_1,x_2+1,x_3,x_4-1)$  (see Fig.  \ref{fig:sigma_omega}). For instance, in the case $\alpha=1$, starting from the single vector with $D=1$ one can obtain all the vectors of c-path lengths for $D=3$: $\sigma(k+1,k,k,k)=(k+1,k+1,k-1,k)$ and $\omega(k+1,k,k,k)=(k+1,k+1,k,k-1)$ (both represent the same vector of c-path lengths making an appropriate rotation). To obtain the other two vectors of $D=3$, just first rotate the initial vector and apply the operations, $\sigma(k,k+1,k,k)=(k,k+2,k-1,k)$ and 
$\omega(k,k+1,k,k)=(k,k+2,k,k-1)$. With these definitions the following assertion holds.

\begin{figure}[htb]
	\begin{center}
		\begin{tabular}{ccc}
			\includegraphics[scale=0.65]{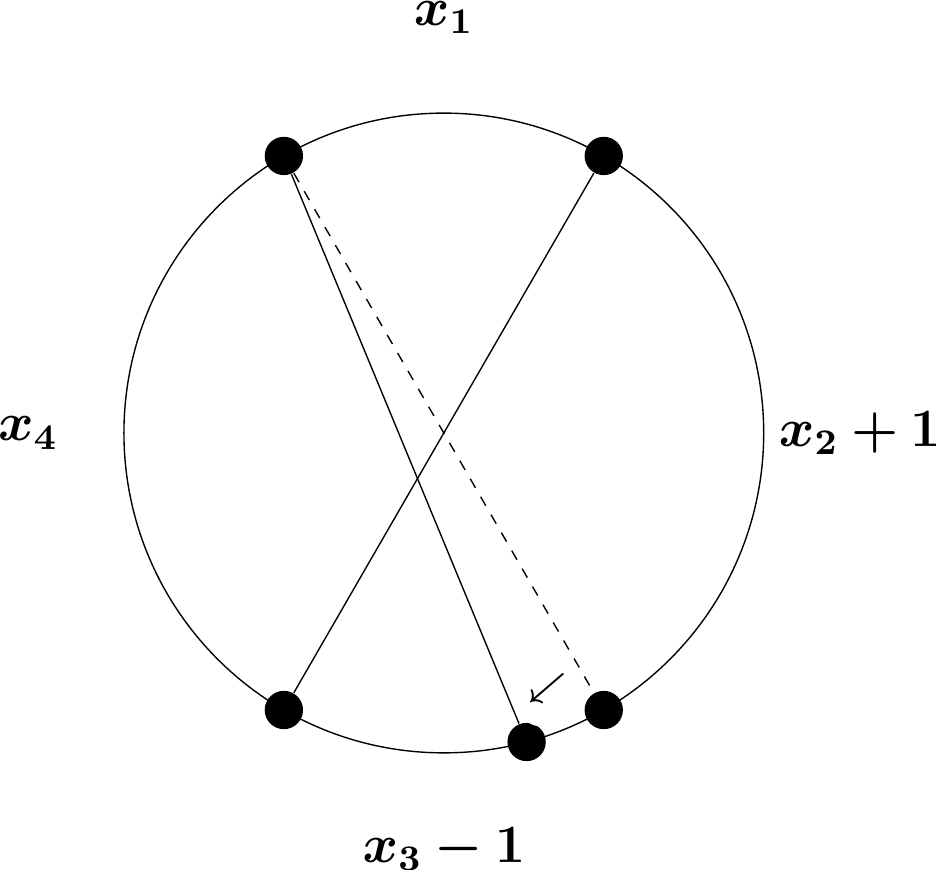} &  & \includegraphics[scale=0.65]{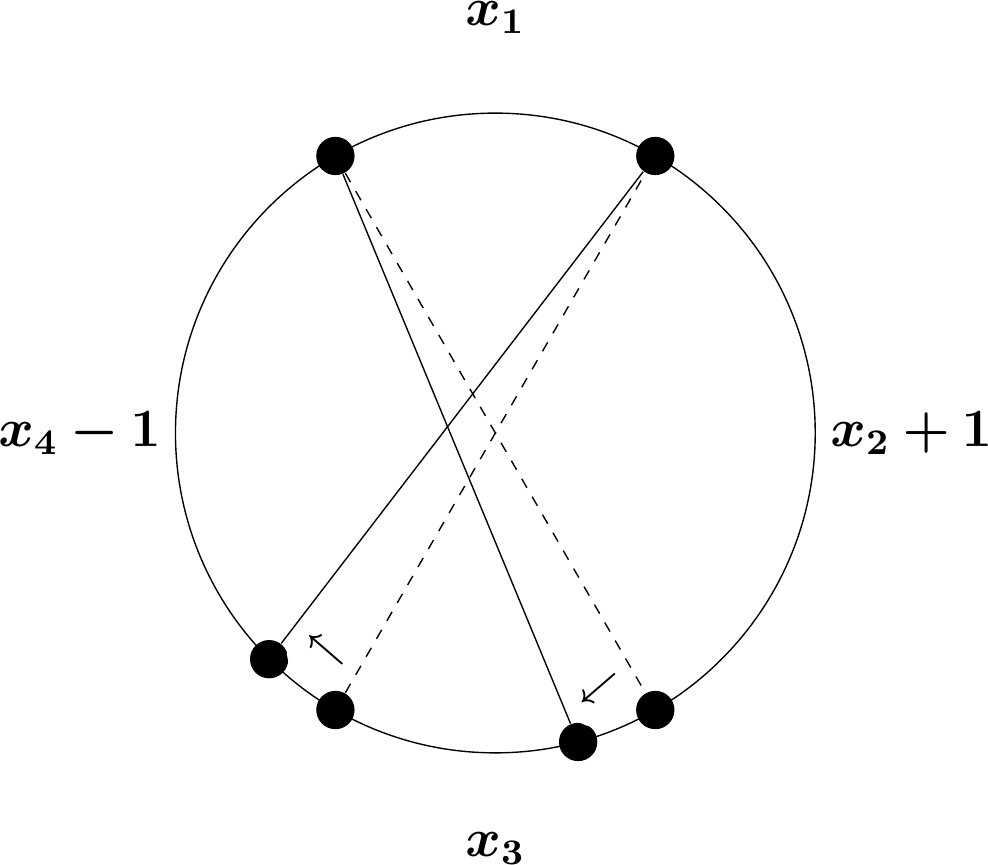} \\
			$\sigma({\bf x})=(x_1,x_2+1,x_3-1,x_4)$ & \hspace{0.2cm} & $\omega({\bf x})=(x_1,x_2+1,x_3,x_4-1)$ \\
		\end{tabular}
	\end{center}
	\caption{(a) Graphical representation of $\sigma({\bf x})$ and $\omega({\bf x})$, where ${\bf x}=(x_1,x_2,x_3,x_4)$}\label{fig:sigma_omega}
\end{figure}

\begin{lemma}\label{lem:lemma}
	Let ${\bf x}=(x_1,x_2,x_3,x_4)\in X_n$ and ${\bf y}=(y_1,y_2,y_3,y_4) \in X_n$ such that $D({\bf y})=D({\bf x})+2$ and either ${\bf y}=\sigma({\bf x})$ or ${\bf y}=\omega({\bf x})$. Then $f({\bf x}) \geq f({\bf y})$ and $g({\bf x}) \geq g({\bf y})$, where $f$ and $g$ are defined by equation \eqref{eq:functfg}.
\end{lemma}
\begin{proof}
	First we suppose that ${\bf y}=\sigma({\bf x})$, that is, $y_1=x_1$, $y_2=x_2+1$, $y_3=x_3-1$ and $y_4=x_4$. From $D({\bf y})=D({\bf x})+2$ we have that 
	\[
	|(x_2+1)-k|-|x_2-k|+ |(x_3-1)-k|-|x_3-k|=2.
	\]
	Which is equivalent to
	\begin{equation}\label{eq:lema}
	\big( |(x_2-k)+1|-|x_2-k| \big)+ \big(|(x_3-k)-1|-|x_3-k| \big)=2.
	\end{equation}
	Taking into account that $|a+1|-|a| \leq 1$ and $|b-1|-|b| \leq 1$ and both equalities hold if and only if $a \geq 0$ and $b \leq 0$, respectively. Then equation \eqref{eq:lema} is equivalent to $x_2 \geq k$ and $x_3 \leq k$. Besides, from equation \eqref{eq:Nm-2} we have $f({\bf x})-f({\bf y})=(x_2-x_3)(x_1+x_4+1)$. From $x_2 \geq k$ and $x_3 \leq k$ we have  $(x_2-x_3)\geq 0$ and hence $f({\bf x})-f({\bf y})=(x_2-x_3)(x_1+x_4+1) \geq 0$ since both factors are positive. Besides, again from equation \eqref{eq:Nm-2}, $g({\bf x})-g({\bf y})=x_2-x_3+1$ which is also positive since $(x_2-x_3)\geq 0$.\\
	In the case when ${\bf y}=\omega({\bf x})$, we can observe that $D({\bf y})=D({\bf x})+2$ is equivalent to $x_2 \geq k$ and $x_4 \leq k$. Besides, $f({\bf x})-f({\bf y})=(x_1+2)x_2+(x_2+1)x_3-(x_1+x_3+2)x_4+x_1+2$. Taking into account $x_2 \geq k$ and $x_4 \leq k$, we have $f({\bf x})-f({\bf y})  =  (x_1+2)x_2+(x_2+1)x_3-(x_1+x_3+2)x_4+x_1+2 \geq (x_1+2)k+(k+1)x_3-(x_1+x_3+2)k+x_1+2 =x_1+x_3+2 \geq 0$. Besides $g({\bf x})-g({\bf y})=x_2-x_4+1$ which is also positive.
\end{proof}

\begin{theorem}\label{th:n+2}
	Let $G$ be a uniformly most reliable graph in ${\cal H}(n,n+2)$. Then, $G$ is of type A with vector of c-path lengths
	\begin{itemize}
		\item $(k,k,k,k)$ if $n=4k$ for some positive integer $k$.
		\item $(k+1,k,k,k)$ if $n=4k+1$ for some positive integer $k$.
		\item $(k+1,k,k+1,k)$ if $n=4k+2$ for some positive integer $k$.
		\item $(k+1,k+1,k+1,k)$ if $n=4k+3$ for some positive integer $k$.
	\end{itemize}
	
\end{theorem}

\begin{proof}By proposition \ref{prop:propAB} we have to take into account only graphs of type A. In order to maximize the coefficients of the reliability polynomial, let us consider the discrete functions $f$ and $g$ defined in \eqref{eq:functfg}. Our goal is to find the maximum of $f$ and $g$ in $X_n$. There is a natural extension of these functions to the continuous side just setting $\hat{f}:\hat{X_n} \rightarrow \mathbb{R}$, $\hat{g}:\hat{X_n} \rightarrow \mathbb{R}$ where $\hat{X_n}=\{(x_1,x_2,x_3,x_4) \in \mathbb{R}^4 \, | \, x_1+x_2+x_3+x_4=n \} \subset \mathbb{R}^4$ and $\hat{f}$ and $\hat{g}$ defined as $f$ and $g$, respectively. \\
	
	First we find the maximum of $\hat{f}$ in $\hat{X_n}$. To this end one can apply the method of Lagrange's multipliers or simply define $\tilde{f}:\mathbb{R}^3 \rightarrow \mathbb{R}$ as ${\tilde f}(x_1,x_2,x_3)=\hat{f}(x_1,x_2,x_3,n-x_1-x_2-x_3-x_4)$ to avoid the restriction given by $\hat{X_n}$, that is
	$$ \begin{array}{lll}
	\tilde{f}(x_1,x_2,x_3)& = &
	{\left(n - x_{1} -
		x_{2} - x_{3}\right)} x_{1} x_{2} + {\left(n - x_{1} - x_{2} -
		x_{3}\right)} x_{1} x_{3} +  \\
	& &
	{\left(n - x_{1} - x_{2} - x_{3}\right)}
	x_{2} x_{3} + x_{1} x_{2} x_{3} + {\left(n - x_{1} - x_{2}\right)}
	{\left(x_{1} + x_{2}\right)} + \\
	& &
	{\left(n - x_{2} - x_{3}\right)}
	{\left(x_{2} + x_{3}\right)} + n.
	\end{array}
	$$
	The critical values of $\tilde{f}(x_1,x_2,x_3)$ are given by the solutions of $\nabla \tilde{f}=(\frac{\partial \tilde{f}}{\partial x_1},\frac{\partial \tilde{f}}{\partial x_2},\frac{\partial \tilde{f}}{\partial x_3})=(0,0,0)$. In our case, 
	$$
	\begin{array}{lll}
	\dfrac{\partial \tilde{f}}{\partial x_1} & = &  
	{\left(n -
		x_{1} - x_{2} - x_{3}\right)} x_{2} - x_{1} x_{2} + {\left(n - x_{1} -
		x_{2} - x_{3}\right)} x_{3} - x_{1} x_{3} + n \\
	& & - 2x_{1} - 2x_{2}. \\
	\dfrac{\partial \tilde{f}}{\partial x_2} & = &  
	{\left(n - x_{1} - x_{2} - x_{3}\right)} x_{1} - x_{1} x_{2} +
	{\left(n - x_{1} - x_{2} - x_{3}\right)} x_{3} - x_{2} x_{3} + 2n \\
	& &
	-2x_{1} - 4x_{2} - 2x_{3}. \\
	\dfrac{\partial \tilde{f}}{\partial x_3} & = &  
	{\left(n - x_{1} - x_{2} -
		x_{3}\right)} x_{1} + {\left(n - x_{1} - x_{2} - x_{3}\right)} x_{2} -
	x_{1} x_{3} - x_{2} x_{3} + n\\
	& &
	-2 \, x_{2} - 2 \, x_{3}.
	\end{array}
	$$
	It happens that $(\frac{n}{4},\frac{n}{4},\frac{n}{4}) \in \mathbb{R}^3$ is the unique solution of $\nabla \tilde{f}=0$ satisfying $x_i  \geq 0$. Moreover, the Hessian matrix $H$ of $\tilde{f}$ at $(\frac{n}{4},\frac{n}{4},\frac{n}{4})$ is
	$$
	H(\frac{n}{4},\frac{n}{4},\frac{n}{4})= - \left(\begin{array}{ccc}
	n + 2 & \frac{n}{2} + 2 & \frac{n}{2} \\
	\frac{n}{2} + 2 & n + 4 & \frac{n}{2} + 2 \\
	\frac{n}{2} & \frac{n}{2} + 2 & n + 2
	\end{array}\right)
	$$
	which is negative definite for all $n \geq 1$. Since $\hat{X_n}$ is a convex set, then $(\frac{n}{4},\frac{n}{4},\frac{n}{4}) \in \mathbb{R}^3$ is a global maximum of $\tilde{f}$. This means that $(\frac{n}{4},\frac{n}{4},\frac{n}{4},\frac{n}{4}) \in \mathbb{R}^4$ is a global maximum of $\hat{f}$. Using a similar reasoning we can see that the same happens with function $\hat{g}$, where also $(\frac{n}{4},\frac{n}{4},\frac{n}{4},\frac{n}{4}) \in \mathbb{R}^4$ is a global maximum of $\hat{g}$. \\
	
	Now we move back to the discrete side of the problem. If $n \equiv 0 \pmod 4$, that is, $n=4k$, then $(\frac{n}{4},\frac{n}{4},\frac{n}{4},\frac{n}{4})=(k,k,k,k) \in X_n$ and hence this is precisely the vector of c-path lengths that gives the maximum value for $\tau$ and $N_{m-2}$ in \eqref{eq:functfg}. The problem becomes more difficult for the remaining cases, that is, when $n=4k+\alpha$, $\alpha \in \{1,2,3\}$. Then $(\frac{n}{4},\frac{n}{4},\frac{n}{4},\frac{n}{4})=(k+\frac{\alpha}{4},k+\frac{\alpha}{4},k+\frac{\alpha}{4},k+\frac{\alpha}{4}) \notin X_n$ and we want to find the vectors of c-path lengths that gives the maximum value of $f$ and $g$ in $X_n$. By Lemma \ref{lem:lemma}, if an element of $X_n$ is a maximum of the functions $f$ and $g$ then it must have minimum $D$, hence we just have to look at those vectors of c-path lengths with minimum $D$. When $\alpha=1$, there is only one vector of c-path lengths with minimum $D$, which is $(k+1,k,k,k)$. \\
	
	For $\alpha=2$, the set of vectors with minimum $D$ is $\{(k+2,k,k,k),(k+1,k+1,k,k),(k+1,k,k+1,k)\}$. According to equation \eqref{eq:Nm-2}, the last two vectors are maximum with the same value for $g$, which is $6k^{2} + 14k + 6$, meanwhile for $f$ we have that $(k+1,k,k+1,k)$ has maximum value $f(k+1,k,k+1,k)=4k^3 + 14k^2 + 14k + 4$. \\
	
	Finally, for $\alpha=3$, the set of vectors $\{(k+3,k,k,k),(k+2,k+1,k,k),(k+2,k,k+1,k),(k+1,k+1,k+1,k)\}$ has minimum $D$. The vector of c-path lengths $(k+1,k+1,k+1,k)$ achieves the maximum both for $f$ and $g$, where $f(k+1,k+1,k+1,k)=4k^3 + 17k^2 + 22k + 8$ and $g(k+1,k+1,k+1,k)=6k^2 + 17k + 10$.
	
\end{proof}

\begin{example}
For $n=11$ and $m=13$, the uniformly most reliable hamiltonian graph must have vector of c-path lengths $(3,3,3,2)$ and $N_{m-2}=68$ and $\tau=152$ (according to Theorem \ref{th:n+2}, case $k=2$ and $\alpha=3$). We also used \textit{Nauty}\footnote{http://users.cecs.anu.edu.au/~bdm/nauty} to generate all non-isomorphic graphs with $11$ vertices and $13$ edges. There are 33851 of such graphs. Then we filtered the hamiltonian ones using a function from a \textit{Python} library called \textit{Graphx}. There are 56 hamiltonian graphs in this case and the uniformly most reliable hamiltonian graph is the expected one. We also computed the uniformly most reliable graph among the total set of graphs (which corresponds to a particular subdivision of $K_4$, as expected). We have depicted both graphs in Figure \ref{fig:rel_max_min}. The computational method to obtain the reliability polynomial is explained in Section \ref{sec:Comp_asp}.
\end{example}

\begin{figure}[htb]
	\begin{center}
		\begin{tabular}{ccc}
			\includegraphics[scale=0.65]{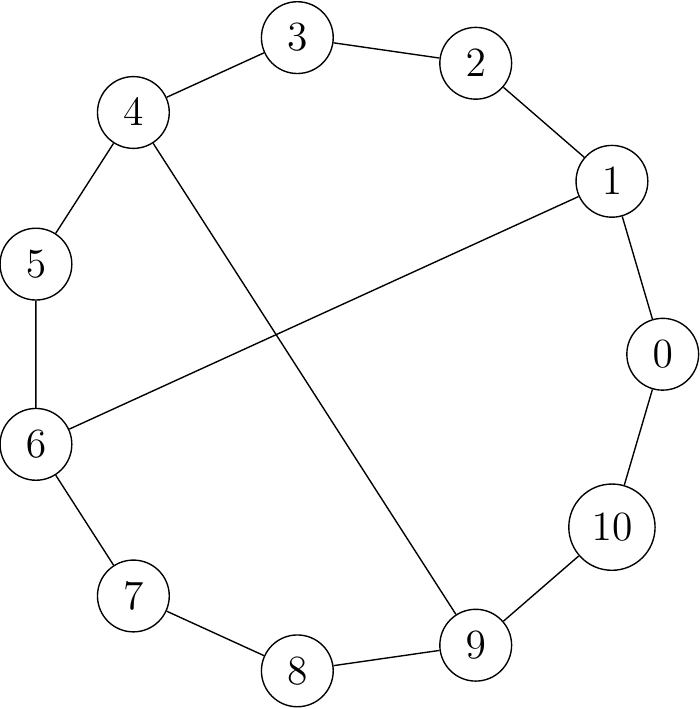} &  & \includegraphics[scale=0.65]{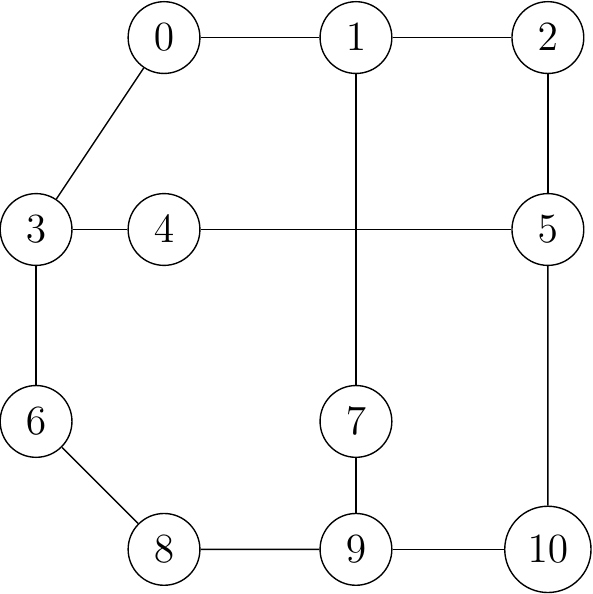} \\
			$N_{m-2}=68$ and $\tau=152$. & \hspace{0.1cm} & $N_{m-2}=70$ and $\tau=160$. \\
		\end{tabular}
	\end{center}
\caption{Uniformly most realiable graphs in ${\cal H}(11,13)$ and ${\cal G}(11,13)$, respectively, and the corresponding coefficients $N_{m-2}$ and $\tau$ of their reliability polynomials.}
	 \label{fig:rel_max_min}
\end{figure}

\subsection{The fair cake-cutting graph $FCG_{n,c}$}\label{sec:faircake}

The uniformly most reliable hamiltonian graphs for $m=n+1$ and $m=n+2$ can be described in terms of a {\em fair cake-cutting process}. There are different definitions for this process in the literature (see \cite{brandt2016handbook,Rom17}). For our purposes, we have to cut a cake (hamiltonian cycle) performing a given number of cuts $c$. The idea is to deliver each part as equal (fair) as possible for every guest. Each cut of the cake is represented in the graph by a diametrical chord, that is, one edge joining two opposite vertices. Given the order $n$ of the hamiltonian cycle and the number of chords $c$ (`cuts' in terms of cake cutting), we present an algorithm to construct what we call {\em The fair cake-cutting graph} $FCG_{n,c}$.  
A detailed explanation of algorithm \ref{algo_faircake} is the following: With the given number of nodes, the algorithm constructs its corresponding cycle. After, the vertices of its hamiltonian cycle are stored into a list. Then,  the `separation' between vertices of degree $3$ (corresponding to the endpoints of the chords) is calculated. This `separation' is also stored in the variable `position' that in the first iteration of the loop will be one of the first chord endpoints. The loop in line 5 places all the cuts to the previously generated cycle. Line 6-9: Each edge chord is added to the graph using the vertices stored inside the hamiltonian cycle list in the corresponding positions. Each position is calculated by adding the previous position plus the separation. The last endpoint of each cut must have maximum distance from its first endpoint, this distance is equal to the half of the nodes of the graph. Notice that the positions of the list must be a natural number but its decimal part is not overlooked. This part is added when the next position is calculated. Finally, the cycle graph with all the chords placed is returned. The idea behind algorithm \ref{algo_faircake} is to balance the distance between the chords.\\
\IncMargin{1em}
\begin{algorithm}[htb]
	\SetKwData{Left}{left}\SetKwData{This}{this}\SetKwData{Up}{up}
	\SetKwFunction{Union}{Union}\SetKwFunction{FindCompress}{FindCompress}
	\SetKwInOut{Input}{input}\SetKwInOut{Output}{output}
	
	\Input{$n \leftarrow$ Number of desired vertices for the $FCG$\\
		$c \leftarrow$ Number of desired chords for the $FCG$}
	\Output{$FCG$}
	\BlankLine
	\BlankLine
	\textbf{\textit{graph}} $fcg \leftarrow$ Cycle with edges $(0,1),(1,2),\dots,(n-1,0)$. \\
	
	\BlankLine
	\textbf{\textit{float}} $separation = n/(2c)$ \\
	\textbf{\textit{float}} $position = separation$ \\
	\textbf{\textit{int}} $placedch = 0$ \\
	\BlankLine
	\While{ $placedch\leq c$}{
		Add edge $(int(position)- 1, int(position+n/2)- 1)$ \\
		//-1 due that the nodes begin with 0\\
		\BlankLine
		$position = position + separation$ //note that the variable is still a float\\
		\BlankLine
		$placedch = placedh + 1$	
	}
	\Return graph 
	\caption{Fair Cake construction}\label{algo_faircake}	
\end{algorithm}\DecMargin{1em}

\begin{example}
For $n=16$ and $c=3$, we start from a cycle graph $C_{16}$ with edge set $\{(0,1),(1,2), \dots, (15,0)\}$. This table summarizes the procedure that the algorithm follows to calculate the positions $(p_1, p_2)$ of the hamiltonian path list for each chord $(v_1, v_2)$:
\begin{table}[htb]
	\begin{center}
		\begin{tabular}{|c|c|c|c|c|c|c|}
			\hline
					\# cut & Separation & Accumulated & Total & $p_1$ & $p_2$ & $(v_1, v_2)$\\ \hline \hline
					1st & $16/6 = 2.\overline{6}$ & 0 & $2.\overline{6}$ & 2 & 2 + 8 = 10 & $(2-1, 10-1)$\\ \hline
					2nd & $16/6 = 2.\overline{6}$ & $2.\overline{6}$ & $5.\overline{3}$ & 5 & 5 + 8 = 13 & $(5-1, 13-1)$\\ \hline
					3rd & $16/6 = 2.\overline{6}$ & $5.\overline{3}$ & 8 & 8 & 8 + 8 = 16 & $(8-1, 16-1)$\\ \hline
			\hline
		\end{tabular}
	\end{center}
	\caption{Algorithm of construction of $FCG_{16,3}$}\label{tab:faircakeexample}
\end{table}
\noindent
First the separation of the cuts is calculated, $\frac{n}{2c} = \frac{16}{6} = 2.\overline{6}$ and then the accumulated part is added. This part is the summation of all previously calculated values (Totals). With this information, the positions of the chord endpoints can be obtained: $p_1$ is equal to the decimal part from the `Total'. $p_2=p_1 + \frac{n}{2}$. Finally, the vertices of each chord are extracted from the hamiltonian path list using the previous positions: For instance, since the nodes starts with 0, at position 5 we have the vertex 4, and the position 13 the vertex 12, therefore we add edge $(4, 12)$.\\

\begin{figure}[htb]
	\begin{center}
		\begin{tabular}{ccc}
			\includegraphics[scale=0.65]{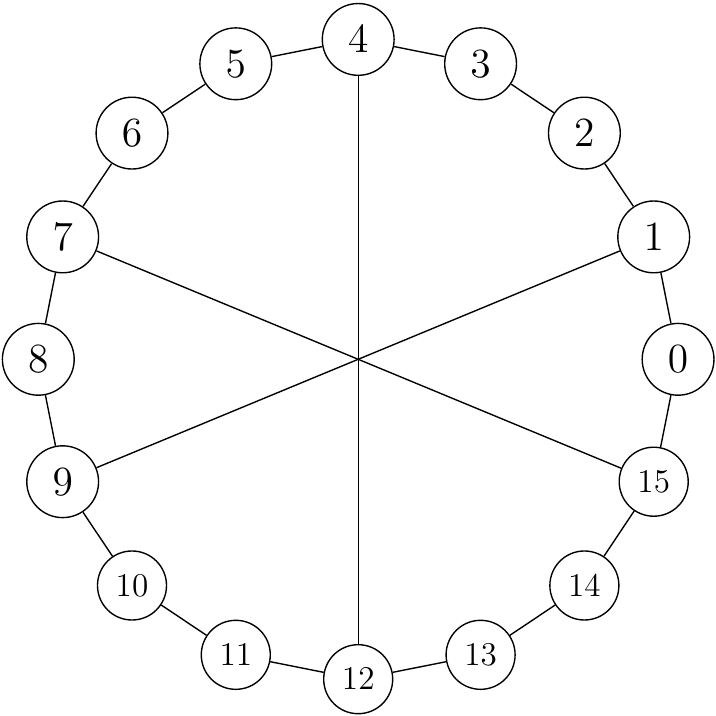} &  & \includegraphics[scale=0.65]{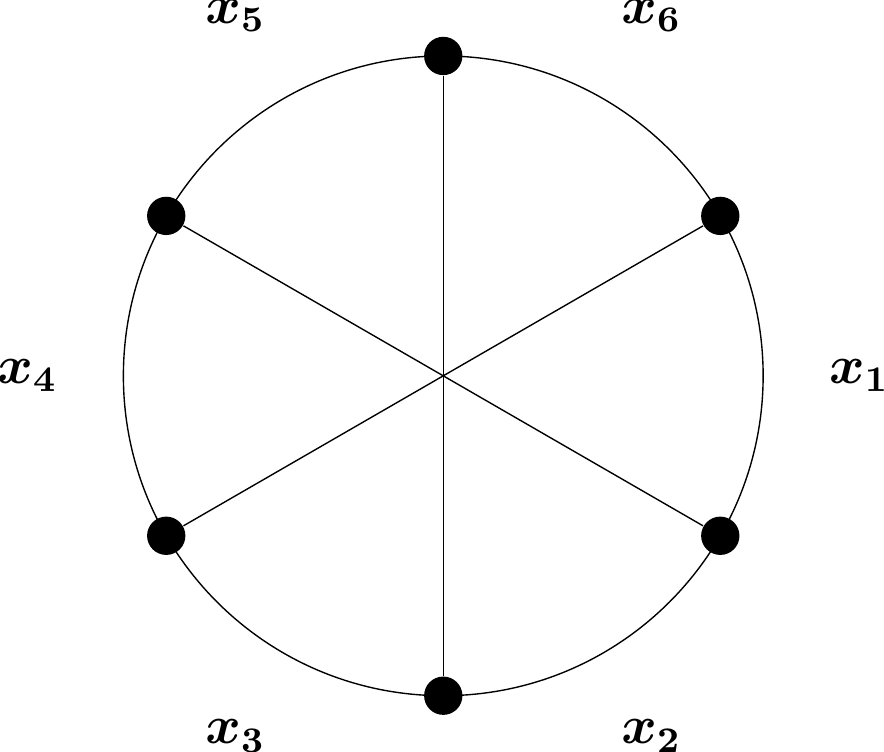} \\
			 (a) & \hspace{0.5cm} & (b) \\
		\end{tabular}
	\end{center}
	\caption{\newline
		(a) The fair cake-cutting graph $FCG_{16,3}$. It has vector of c-path lengths $(2,3,3,2,3,3)$.\newline 
		(b) A graph in ${\cal H}(n,n+3)$ of type A with diametrical chords.}\label{fig:fc16}
	\end{figure}
\end{example}

The fair cake-cutting graphs $FCG_{n,c}$ are similar to the family of graphs that are a solution of the following augmentation problem:  Starting from the cycle graph $C_n$, add a single edge at each step, in order to maximize the reliability of the resulting graph. Romero (see \cite{Rom17,Rom18}) finds the sequence of graphs $\{G^{(i)}\}_{i=0,\dots, \lfloor \frac{n}{2} \rfloor}$ with $G^{(0)}=C_n$ such that $G^{(i+1)}=G^{(i)}\cup \{e_{i+1}\}$ gives the best augmentation. This process (called also the fair cake-cutting process) ends with the circulant graph with steps $1$ and $\frac{n}{2}$, that is, a cubic hamiltonian graph where every vertex is joined to its opposite vertex in the hamiltonian cycle. For $n=8$, this is the Wagner graph depicted in Figure \ref{fig:mrg}, which is a uniformly most reliable graph for $(8,12)$. The main difference between both families is that in our case the number of cuts $c$ is previously known, meanwhile, in this other family, the cuts are performed as the guest arrives. In fact, both families differ when three or more cuts are performed. As a consequence of Theorem \ref{th:n+2} and Proposition \ref{prop:n+1} we have that $FCG_{n,c}$ produces uniformly most reliable hamiltonian graphs for $c=1$ and $c=2$.

\begin{corollary}
$FCG_{n,1}$ and $FCG_{n,2}$ are uniformly most reliable hamiltonian graphs for $m = n+1$ and $m=n+2$, respectively.
\end{corollary}

\section{Computational approach for $m=n+3$ and beyond}\label{sec:Comp_asp}

Uniformly most reliable hamiltonian graphs have been totally characterized when $m \leq n+2$ in section \ref{sec:most_reliable_ham}. We also give a construction of these optimal graphs in section \ref{sec:faircake}. Beyond this point, we have performed some computational tools in order to generate uniformly most reliable hamiltonian graphs for $m \geq n+3$. First, we discuss the computation of the reliability polynomial of a graph. 

\subsection*{The computation of the reliability polynomial}

\textit{The factoring theorem} is a recursive method for computing $\mbox{Rel}(G,p)$ based on the combination of two graph operations: edge deletion $G-e$, and edge contraction $G/e$ (for further details see, for instance, \cite{Perez2018}):

\begin{equation}
\mbox{Rel}(G,p) = 
\begin{cases}
\mbox{Rel}(G-e, p) & \mbox{ if } e \mbox{ is a loop, } \\
p \mbox{Rel}(G/e, p) & \mbox{ if } e \mbox{ is a cut-edge, } \\
(1-p)\mbox{Rel}(G-e,p) + p\mbox{Rel}(G/e,p) & \mbox{ otherwise. } 
\end{cases}
\label{eq:deletion-contraction}
\end{equation}

The algorithm \ref{alg:rel_poly_ft} shows the implementation of the theorem done in our code. \\
\vspace{0.2cm}
\IncMargin{1em}
\begin{algorithm}[htb]
	\SetKwData{Left}{left}\SetKwData{This}{this}\SetKwData{Up}{up}
	\SetKwFunction{Union}{Union}\SetKwFunction{FindCompress}{FindCompress}
	\SetKwInOut{Input}{input}\SetKwInOut{Output}{output}
	
	\Input{$g \leftarrow$ Graph}
	\Output{Reliability Polynomial}
	
	\BlankLine
	// If the graph is not connected, then it has a reliability polynomial of 0\\
	\If{ g is not connected}{
		\Return 0
	}
	\BlankLine
	// if the number of edges $>$ 0, then we perform the two sub-cases of the Factoring Theorem
	\If{ number of edges of g $>$ 0}{
		\textbf{\textit{edge}} e = g.random\_edge(e)\\
		\BlankLine
		\textbf{\textit{graph}} contracted = g.contract\_edge(e)\\
		\textbf{\textit{graph}} deleted = g.delete\_edge(e)\\
		\BlankLine
		\textbf{\textit{polynomial}} rec\_contracted = recursion wih the graph \textit{contracted}\\
		\textbf{\textit{polynomial}} rec\_deleted = recursion wih the graph \textit{deleted}\\
		\textbf{\textit{polynomial}} s = $p \cdot rec\_contracted + (1-p) \cdot rec\_deleted$\\
		\BlankLine
		\Return \textit{s} 
	}
	\BlankLine
	// Otherwise, we only have 0 edges and 1 vertex, which is connected, so we return 1.\\
	\Return 1
	\caption{Reliability Polynomial Factoring Theorem}\label{alg:rel_poly_ft}	
\end{algorithm}\DecMargin{1em}
\vspace{0.3cm}

We perform a modified version of this \textit{Factoring Theorem} which works slightly different: In each recursion, if there exist some method that can directly retrieve the reliability polynomial or with less cost than another recursion, then, the method will retrieve it and the recursion will stop in that generated subgraph. In other words the main idea to improve this algorithm is to prevent it to `dismantle' the graph to its very basic components (trivial graphs) by giving the reliability of the subgraphs before becoming basic components. We developed a series of fast formulas for specific families of graphs such as multi-tree, multi-cycle and glued cycles that when one of the subgraphs matches one of the families of our formulas then, the reliability is directly returned. With this modified method we have been able to compute the reliability polynomial of all graphs in ${\cal H}(n,n+3)$ for any $n \leq 11$. All (non-isomorphic) graphs have been generated first using {\it Nauty} and we get the hamiltonian ones using the library \textit{Graphx} from \textit{Python}. The drawback of this function is that, at the worst case, runs in linear time ($O(n)$). The coefficients list $N_i$ of the reliability polynomial of those uniformly most reliable hamiltonian graphs is presented in Table \ref{tab:results_table1}.\\

\begin{table}[htb]
	\begin{center}
		\begin{tabular}{|c|c|c|c|c|c|c|}
			\hline
			\ Study case & $Rel(H,p)$ coefficients vector $(N_i)$ \\ \hline \hline
			${\cal H}(6,9)$ & $[1, 9, 36, 78, 81]$ \\ \hline
			
			${\cal H}(7,10)$ & $[1, 10, 44, 104, 117]$ \\ \hline
			
			${\cal H}(8,11)$ & $[1, 11, 53, 137, 168]$ \\ \hline
			
			${\cal H}(9,12)$ & $[1, 12, 63, 178, 240]$ \\ \hline
			
			${\cal H}(10,13)$ & $[1, 13, 74, 226, 328]$ \\ \hline
			
			${\cal H}(11,14)$ & $[1, 14, 86, 284, 445]$ \\ \hline
			\hline
		\end{tabular}
	\end{center}
	\caption{Coefficients list of uniformly most reliable hamiltonian graphs for $6 \leq n \leq 11$ and $m=n+3$.}\label{tab:results_table1}
\end{table}

Every graph listed in Table \ref{tab:results_table1} is of type A (in the sense explained in section \ref{sec:n+2}) and it has some diametrical chords (chords joining two vertices at maximum distance in the hamiltonian cycle, see figure \ref{fig:fc16}). Although we do not have a proof that uniformly most reliable hamiltonian graphs must be of type A with diametrical chords for $m=n+3$, this experimental result encourage us to look optimal graphs in this subset of hamiltonian graphs that we denote as $H_D(n,m)$. So, we designed a direct way to construct all graphs in $H_D(n,n+3)$ that lead us to go beyond $n=11$ vertices: starting with a cycle, draws a diametrical chord dividing the cycle in two parts. Then, makes a set of 2-element combinations between each group of nodes from each part. Notice that all the edges in this set cross the first diametrical chord. With this set it makes another 2-element combinations, but this time with the elements inside the created edge set. Finally, with each combination of 2 edges creates a new graph by adding them into the graph with the diametrical chord. The algorithm \ref{alg:dmtrl_opt} shows in detail this process.\\

\IncMargin{1em}
\begin{algorithm}[htb]
	\SetKwData{Left}{left}\SetKwData{This}{this}\SetKwData{Up}{up}
	\SetKwFunction{Union}{Union}\SetKwFunction{FindCompress}{FindCompress}
	\SetKwInOut{Input}{input}\SetKwInOut{Output}{output}
	
	\Input{$n \leftarrow$ Number of vertices.}
	\Output{List of type \textit{A} hamiltonian graphs with three chords (at least one chord is diametrical).}
	\BlankLine
	\BlankLine
	\textbf{\textit{graph}} $cycle \leftarrow$ Cycle with edges $(0,1),(1,2),\dots,(n-1,0)$. \\
	\textbf{\textit{list}} $vertices \leftarrow cycle.nodes$ \\
	\BlankLine
	// Set the diametrical chord\\
	$cycle\leftarrow$ add edge (0, floor(n/2))
	\BlankLine
	// Remove the nodes of the added chord from \textit{vertices}\\
	$vertices \leftarrow$ remove (0, floor(n/2))\\
	\BlankLine
	// Get possible chords\\
	\textbf{\textit{list}} $hvertices1 \leftarrow$ 1st half of the list \textit{vertices}\\
	\textbf{\textit{list}} $hvertices2 \leftarrow$ 2nd half of the list \textit{vertices}\\
	\textbf{\textit{list}} $possibleVertices \leftarrow$ all combinations of elements between \textit{hvertices1} and \textit{hvertices2}\\
	\BlankLine
	// From all non-existent possible edges, get combinations of 2 chords\\
	// Notice that we already added 1 chord (the diametrical one) \\
	\textbf{\textit{list}} $edgeCombinations \leftarrow$ all combinations of 2 elements from the list \textit{possibleVertices}
	\BlankLine
	// For each combination create a graph and save it into a list of graphs\\
	\textbf{\textit{list}} \textit{hamiltonians}\\
	\For{ combination \textbf{in} combinations}{
		\textbf{\textit{graph}} $tmp \leftarrow$ copy(\textit{cycle})\\
		$tmp\leftarrow$ add edges in \textit{combination}\\
		$hamiltonians \leftarrow$ add \textit{tmp}\\	
	}
	
	\Return \textit{hamiltonians} 
	\caption{Generation of all graphs in ${\cal H_D}(n,n+3)$.}\label{alg:dmtrl_opt}	
\end{algorithm}\DecMargin{1em}
\vspace{0.3cm}
The generated list of graphs contains many isomorphic graphs. Then we use \textit{nauty} to remove isomorphic graphs from the list. We compute the reliability polynomial of each graph of this shorter list with the method explained at the first part of this section. Table \ref{tab:results_table2} shows the uniformly most reliable graphs of type A with at least one diametrical chord for $n \leq 34$ and $m=n+3$. \\

\begin{table}[htb]
	\begin{minipage}{.5\textwidth}
		\begin{tabular}{|c|c|c|c|c|c|c|}
			\hline
			\ Study case & $Rel(H, p)$ coefficients vector $(N_i )$ \\ \hline \hline
			${\cal H_D}(12,15)$ & $[1, 15, 99, 353, 600]$ \\ \hline
			
			${\cal H_D}(13,16)$ & $[1, 16, 112, 422, 755]$ \\ \hline
			
			${\cal H_D}(14,17)$ & $[1, 17, 126, 502, 948]$ \\ \hline
			
			${\cal H_D}(15,18)$ & $[1, 18, 141, 594, 1188]$ \\ \hline
			
			${\cal H_D}(16,19)$ & $[1, 19, 157, 697, 1464]$ \\ \hline
			
			${\cal H_D}(17,20)$ & $[1, 20, 174, 814, 1799]$ \\ \hline
			
			${\cal H_D}(18,21)$ & $[1, 21, 192, 946, 2205]$ \\ \hline
			
			${\cal H_D}(19,22)$ & $[1, 22, 210, 1078, 2611]$ \\ \hline
			
			${\cal H_D}(20,23)$ & $[1, 23, 229, 1225, 3088]$ \\ \hline
			
			${\cal H_D}(21,24)$ & $[1, 24, 249, 1388, 3648]$ \\ \hline
			
			${\cal H_D}(22,25)$ & $[1, 25, 270, 1566, 4272]$ \\ \hline
			
			${\cal H_D}(23,26)$ & $[1, 26, 292, 1762, 4995]$ \\ \hline
		\end{tabular}
	\end{minipage}%
	\begin{minipage}{.5\textwidth}
		\begin{tabular}{|c|c|c|c|c|c|c|}
			\hline
			\ Study case & $Rel(H, p)$ coefficients vector $(N_i )$ \\ \hline \hline
			${\cal H_D}(24,27)$ & $[1, 27, 315, 1977, 5832]$ \\ \hline
			
			${\cal H_D}(25,28)$ & $[1, 28, 338, 2192, 6669]$ \\ \hline
			
			${\cal H_D}(26,29)$ & $[1, 29, 362, 2426, 7620]$ \\ \hline
			
			${\cal H_D}(27,30)$ & $[1, 30, 387, 2680, 8700]$ \\ \hline
			
			${\cal H_D}(28,31)$ & $[1, 31, 413, 2953, 9880]$ \\ \hline
	
			${\cal H_D}(29,32)$ & $[1, 32, 440, 3248, 11209]$ \\ \hline
			
			${\cal H_D}(30,33)$ & $[1, 33, 468, 3566, 12705]$ \\ \hline
			
			${\cal H_D}(31,34)$ & $[1, 34, 496, 3884, 14201]$ \\ \hline
	
			${\cal H_D}(32,35)$ & $[1, 35, 525, 4225, 15864]$ \\ \hline
			
			${\cal H_D}(33,36)$ & $[1, 36, 555, 4590, 17712]$ \\ \hline
			
			${\cal H_D}(34,37)$ & $[1, 37, 586, 4978, 19704]$ \\ \hline
			
			\hline
		\end{tabular}
	\end{minipage}

	\caption{Coefficients list of uniformly most reliable hamiltonian graphs of type A with at least one diametrical chord.}\label{tab:results_table2}
\end{table}

Our conjecture about the uniformly most reliable hamiltonian graph in $H(n,m)$ must be in $H_D(n,m)$ is no longer true for $m \geq n+4$. In Figure \ref{fig:mplus4} we present some uniformly most reliable hamiltonian graphs for $m=n+4$ found by computer. Despite the case $n=8,m=12$ (see Figure \ref{fig:mrg} (a)), the remaining cases are not of type A. The situation for $m=n+5$ is similar (see Figure \ref{fig:mplus5}). 

We have been able to find all uniformly most reliable hamiltonian graphs up to $n=11$ vertices and $m=16$ edges using the general method described at the beginning of the section: First we generate a list containing all non-isomorphic graphs of a given order and size using {\em Nauty}. Afterwards, the reliability polynomial of every graph in the list is computed by using our improved version of the {\em factoring theorem}. 

\begin{figure}[htb]
\begin{center}
\begin{tabular}{ccc}
\includegraphics[scale=0.65]{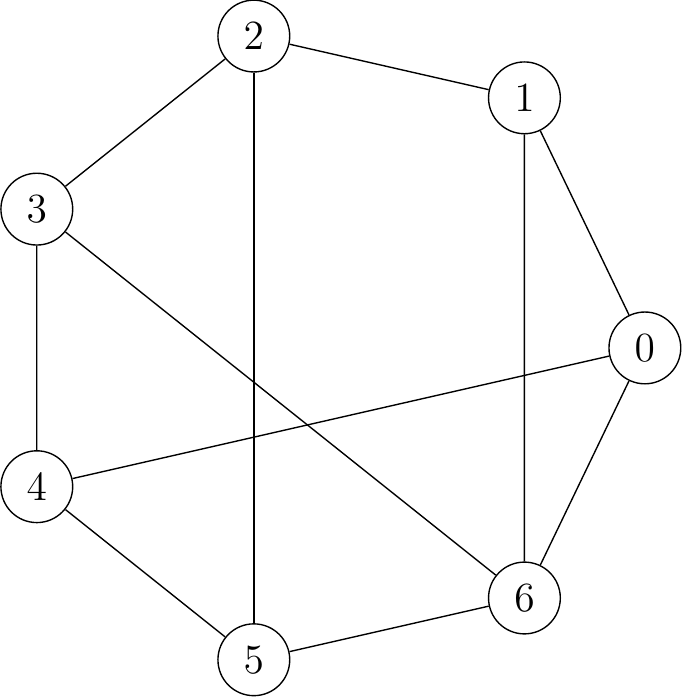} & \includegraphics[scale=0.65]{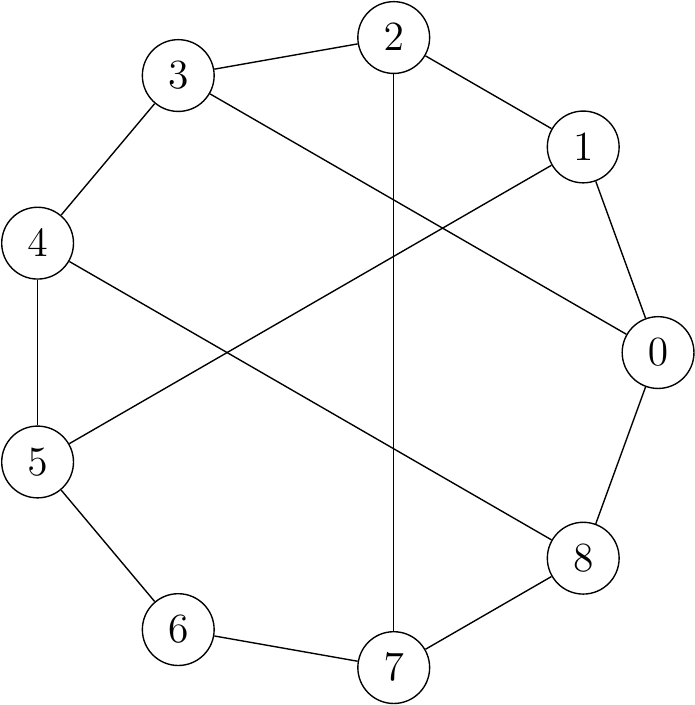} & \includegraphics[scale=0.65]{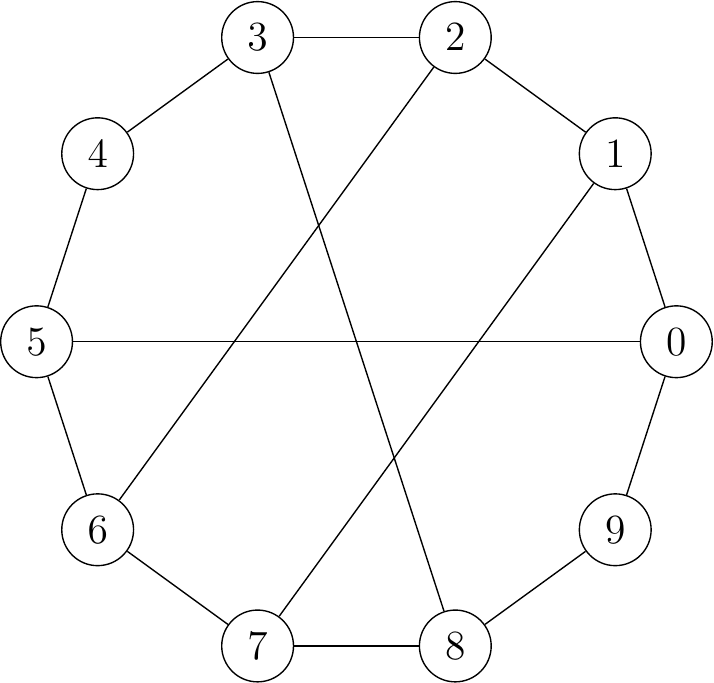} \\
 (a) $n=7,m=11$ & (b) $n=9,m=13$ & (c) $n=10,m=14$ \\
\end{tabular}
\end{center}
\caption{Some uniformly most reliable hamiltonian graphs for $m=n+4$.}\label{fig:mplus4}
\end{figure}

\begin{figure}[htb]
\begin{center}
\begin{tabular}{ccc}
\includegraphics[scale=0.65]{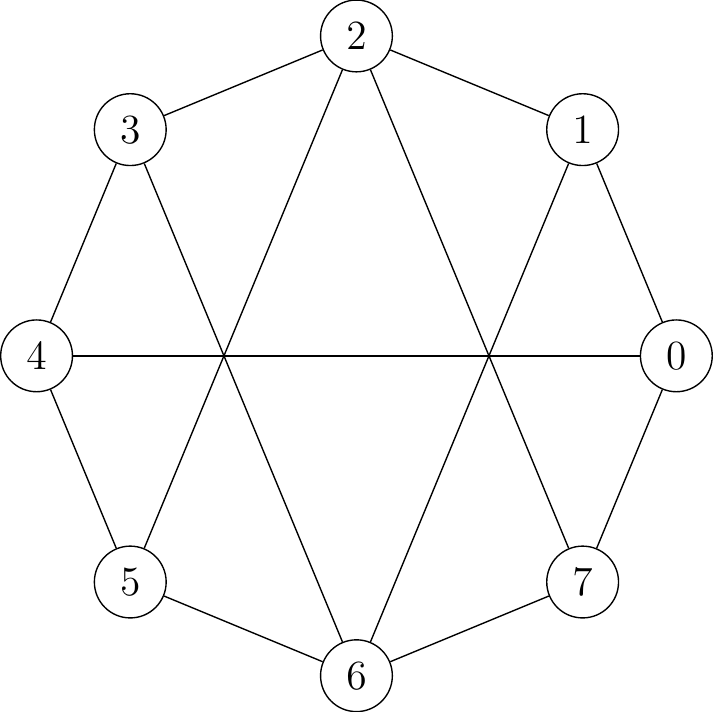} & \includegraphics[scale=0.65]{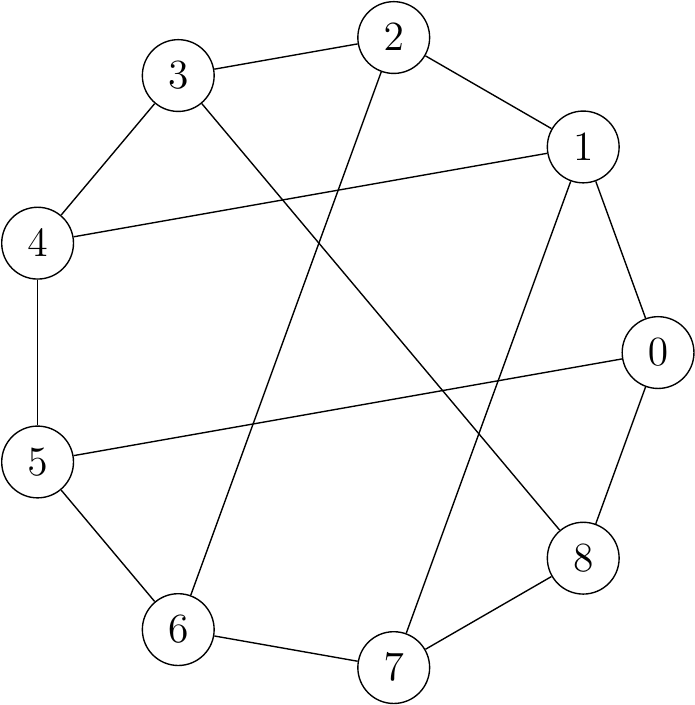} & \includegraphics[scale=0.65]{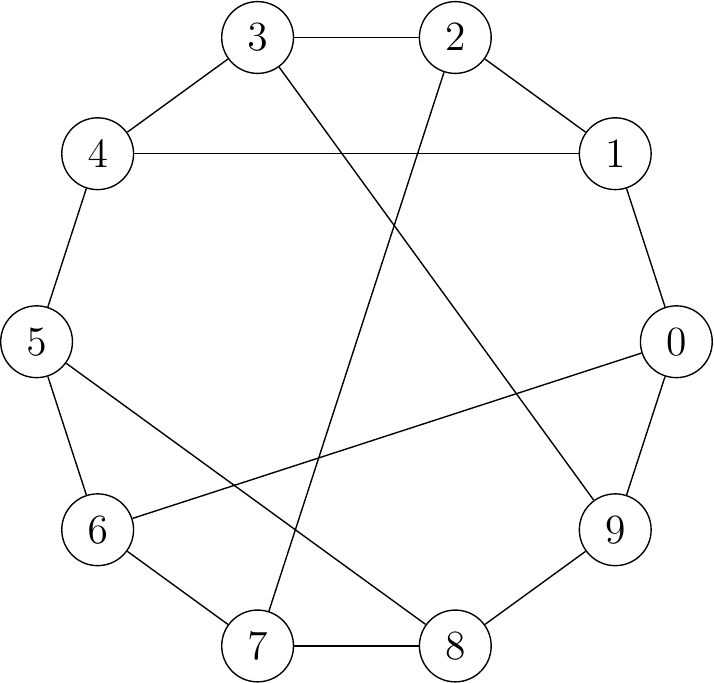} \\
 (a) $n=8,m=13$ & (b) $n=9,m=14$ & (c) $n=10,m=15$ \\
\end{tabular}
\end{center}
\caption{Some uniformly most reliable hamiltonian graphs for $m=n+5$.}\label{fig:mplus5}
\end{figure}

A (simple) graph of order $n$ has at most $m=\binom{n}{2}$ edges. There is only one of such graphs with maximum number of edges, which is the complete graph $K_n$ and hence it is uniformly most reliable. There is also only one graph (up to isomorphisms) with $m=\binom{n}{2}-1$ edges, but for $m=\binom{n}{2}-2$ we have two different graphs: we can eliminate from $K_n$ either a path of length two or two independent edges. This latter case gives the uniformly most reliable graph. More in general, it is already known that when $m \geq \binom{n}{2} - \frac{n}{2}$ then there exist a uniformly most reliable graph which is the graph whose complement graph have a set of independent edges (see \cite{Shier74max}). Any graph with a number of edges large enough is hamiltonian, and hence uniformly most reliable graphs are hamiltonian in this case. For instance, it is well known that every graph satisfying $m \geq 1/2(n^2-3n+6)$ is hamiltonian (see \cite{Chartrand2015}). Every graph with $m \geq \binom{n}{2} - \frac{n}{2}$ satisfies also $m \geq 1/2(n^2-3n+6)$ whenever $n \geq 6$ and hence uniformly most reliable graphs in this case (graphs whose complement has a set of independent edges) are 
also uniformly most reliable hamiltonian graphs.

\begin{proposition}\label{prop:mgran}
Given an integer $n\geq 6$, there is a uniformly most reliable graph in ${\cal H}(n,m)$ for any $m \geq \binom{n}{2} - \frac{n}{2}$.
\end{proposition}

\section{Non existence of uniformly most reliable hamiltonian graphs for some cases}

There are some cases where uniformly most reliable graphs do not exist: it is shown in \cite{Myr91} that the graph $G_2$ of order $n\geq 6$ even, which is defined as the complement of the graph $P_4 \cup K_2 \cup \frac{n-6}{2}K_2$  satisfies $\mbox{Rel}(G_2,p)>\mbox{Rel}(G',p)$ for all $G' \in {\cal G}(n,m)$, $m=\binom{n}{2}-\frac{n+2}{2}$ and $p$ sufficiently close to one. Besides, the complement of $2P_3 \cup \frac{n-6}{2}K_2$, defined as $G_1$, satisfies $\mbox{Rel}(G_1,p)>\mbox{Rel}(G_2,p)$ for $p$ sufficiently close to zero. As a consequence, there is no uniformly most reliable graphs in ${\cal G}(n,m)$ for $m=\binom{n}{2}-\frac{n+2}{2}$. We use this result to prove that there are infinitely many pairs $(n,m)$ where uniformly most reliable hamiltonian graphs do not exist.

\begin{proposition}\label{prop:nonexist}
There are no uniformly most reliable graphs in ${\cal H}(n,m)$ for
\begin{itemize}
 \item $m=\binom{n}{2}-\frac{n+2}{2}$ for all $n \geq 6$ even;
 \item $m=\binom{n}{2}-\frac{n+5}{2}$ for all $n \geq 7$ odd.
\end{itemize}
\end{proposition}
\begin{proof}
For the case $m=\binom{n}{2}-\frac{n+2}{2}$ it is suffice to show that the graphs $G_1$ and $G_2$ defined above as the complements of the graphs $2P_3 \cup \frac{n-6}{2}K_2$ and $P_4 \cup K_2 \cup \frac{n-6}{2}K_2$ for $n \geq 6$ even, respectively, are hamiltonian graphs. To this end notice that the minimum degree for a vertex in $G_1$ is $n-3$. Hence for any two given vertices $u$ and $v$, $d(u)+d(v) \geq 2n-6 \geq n$ for all $n \geq 6$. Applying Theorem \ref{the:Ore} we deduce that $G_1$ is hamiltonian. The same argument applies for $G_2$. \\

For $n \geq 7$ odd, define $G_3$ as the complement of $C_3 \cup P_4 \cup \frac{n-7}{2}K_2$ and $G_4$ as the complement of $C_5 \cup K_2 \cup \frac{n-7}{2}K_2$ as in \cite{Myr91} and apply again Theorem \ref{the:Ore} to proof that they belong to ${\cal H}(n,m)$, where $m=\binom{n}{2}-\frac{n+5}{2}$. In \cite{Myr91} it is shown that $\mbox{Rel}(G_4,p)>\mbox{Rel}(G',p)$ for all $G' \in {\cal G}(n,m)$ for $p$ sufficiently close to one. Besides $\mbox{Rel}(G_3,p)>\mbox{Rel}(G_4,p)$ for $p$ sufficiently close to zero, and the proof is completed.
\end{proof}


\section{Conclusions and open problems}

In this paper uniformly most reliable hamiltonian graphs have been characterized for $m \leq n+2$ and we give some light for the case $m=n+3$ which somehow follow the previous cases. The situation is different for $m \geq n+4$, where the problem is totally open except for some small cases found by computer. Nevertheless, one can use these results to obtain uniformly most reliable hamiltonian graphs for a fixed value of $n$. For instance, when $n=6$ we have nine cases to analyze ($6 \leq m \leq 15$): For $m=6$, $C_6$ is the uniformly most reliable hamiltonian graph. For $7 \leq m \leq 9$ we have that $FCG_{6,c}$ are uniformly most reliable hamiltonian graphs for $c=1,2,3$, respectively (also $K_{3,3}$ for $m=9$, which is isomorphic to $FCG_{6,3}$). For $m=11$ there is no uniformly most reliable hamiltonian graph (Proposition \ref{prop:nonexist}) and for $m \geq 12$ subgraphs of $K_6$ whose complement graph has a set of independent edges are uniformly most reliable (Proposition \ref{prop:mgran}). It remains the case $m=10$. We have found by computer that adding any edge to the graph $K_{3,3}$ produces the uniformly most reliable graph in this case.

\begin{problem}
Characterize uniformly most reliable hamiltonian graphs for other values of $n$ and $m$.
\end{problem}
The fair cake-cutting graph $FCG_{n,c}$ presented in section \ref{sec:faircake} is a uniformly most reliable graph for many values, but we believe that is also optimal for other values.

\begin{conjecture}
Let  $n \equiv 0 \pmod{2c}$. Then $FCG_{n,c}$ is a uniformly most reliable hamiltonian graph for $m=n+c$.
\end{conjecture}

The problem of finding uniformly most reliable graphs seems difficult, even for restricted versions of the problem, like the one we present here for hamiltonian graphs. It would be worth to study this problem for other classes of graphs.

\begin{problem}
Study the problem of finding uniformly most reliable graphs for a named class of graphs, such as bipartite graphs, Cayley graphs, etc.
\end{problem}

\subsection*{Acknowledgments}
Research of the authors was supported in part by grant MTM2017-86767-R (Spanish Ministerio de Ciencia e Innovacion).

\bibliographystyle{unsrt}
\bibliography{ReliabilityBiblio}

\begin{thebibliography}{10}

\bibitem{Gould2003}
Ronald J.~Gould.
\newblock Advances on the hamiltonian problem - a survey.
\newblock {\em Graphs and Combinatorics}, 19:7--52, 2003.

\bibitem{Perez2018}
Hebert Perez-Roses.
\newblock Sixty years of network reliability.
\newblock {\em Mathematics in Computer Science}, 12(3):275--293, 6 2018.

\bibitem{Boesch88}
F.~T. Boesch.
\newblock On the synthesis of optimally reliable networks having unreliable
  nodes but reliable edges.
\newblock In {\em INFOCOM '88. Networks: Evolution or Revolution, Proceedings.
  Seventh Annual Joint Conference of the IEEE Computer and Communcations
  Societies, IEEE}, pages 829--834, Mar 1988.

\bibitem{Boesch91}
F.~T. Boesch, X.~Li, and C.~Suffel.
\newblock On the existence of uniformly optimally reliable networks.
\newblock {\em Networks}, 21(2):181--194, 1991.

\bibitem{Gross98}
D.~Gross and J.~T. Saccoman.
\newblock Uniformly optimally reliable graphs.
\newblock {\em Networks}, 31(4):217--225, 1998.

\bibitem{Smith90}
Derek~H. Smith and Lynne~L. Doty.
\newblock On the construction of optimally reliable graphs.
\newblock {\em Networks}, 20(6):723--729, 1990.

\bibitem{Brown07}
J.I. Brown and Xiaohu Li.
\newblock Uniformly optimal digraphs for strongly connected reliability.
\newblock {\em Networks}, 49(2):145--151, 2007.

\bibitem{Colbourn87}
Charles~J. Colbourn.
\newblock {\em The Combinatorics of Network Reliability}.
\newblock Oxford University Press, New York, NY, USA, 1987.

\bibitem{Myr91}
Wendy Myrvold, Kim~H. Cheung, Lavon~B. Page, and Jo~Ellen Perry.
\newblock Uniformly-most reliable networks do not always exist.
\newblock {\em Networks}, 21(4):417--419, 1991.

\bibitem{WangWu74}
J.~F. Wang and M.~H. Wu.
\newblock Network reliability analysis: on maximizing the number of spanning
  trees.
\newblock {\em Proceeding of the National Science Council, Republic of China,
  Part A, Physical Science and Engineering}, 11(3):193--196, 1987.

\bibitem{Rom17}
P.~{Romero}.
\newblock Building uniformly most-reliable networks by iterative augmentation.
\newblock In {\em 2017 9th International Workshop on Resilient Networks Design
  and Modeling (RNDM)}, pages 1--7, Sep. 2017.

\bibitem{Satya92}
A.~Satyanarayana, L.~Schoppmann, and C.~L. Suffel.
\newblock A reliability-improving graph transformation with applications to
  network reliability.
\newblock {\em Networks}, 22(2):209--216, 1992.

\bibitem{Rom18}
Guillermo Rela, Franco Robledo, and Pablo Romero.
\newblock Petersen graph is uniformly most-reliable.
\newblock In {\em Machine Learning, Optimization, and Big Data}, pages
  426--435, Cham, 2018. Springer International Publishing.

\bibitem{brandt2016handbook}
F.~Brandt, V.~Conitzer, U.~Endriss, J.~Lang, and A.D. Procaccia.
\newblock {\em Handbook of Computational Social Choice}.
\newblock Cambridge University Press, 2016.

\bibitem{Shier74max}
Douglas~R. Shier.
\newblock Maximizing the number of spanning trees in a graph with n nodes and m
  edges.
\newblock {\em Journal of research of the national bureau of standards B.
  Mathematical Sciences}, 78B(4):193--196, 1974.

\bibitem{Chartrand2015}
Gary Chartrand, Linda Lesniak, and Ping Zhang.
\newblock {\em Graphs \& Digraphs}.
\newblock Chapman \& Hall, 6th edition, 2015.

\end{thebibliography}

\end{document}